\documentclass[12pt]{amsart}

\usepackage[english]{babel}
\usepackage{amsmath,amsthm}
\usepackage{amsfonts}
\usepackage{graphicx}
\usepackage{epsfig}
\usepackage{epstopdf}
\usepackage{url}
\usepackage{amssymb}
\usepackage{color}
\usepackage[makeroom]{cancel}
\usepackage{tikz-cd}
\usepackage{tikz}
\usetikzlibrary{matrix,calc}
\usepackage[titletoc]{appendix}
\usepackage{multirow}
\usepackage{tabularx}
\usepackage{bigdelim}
\usepackage{blkarray}
\usepackage{bbm}

\usepackage[%
breaklinks%
,colorlinks%
,linkcolor=red%
,anchorcolor=blue%
,pagecolor=blue%
,citecolor=blue%
,bookmarks=false%
]{hyperref}

\definecolor{codegreen}{rgb}{0,0.6,0}
\definecolor{codegray}{rgb}{0.5,0.5,0.5}
\definecolor{codepurple}{rgb}{0.58,0,0.82}
\definecolor{backcolour}{rgb}{0.95,0.95,0.92}
\usepackage{listings}
\lstdefinelanguage{Sage}[]{Python}
{morekeywords={False,sage,True},sensitive=true}
\definecolor{dblackcolor}{rgb}{0.0,0.0,0.0}
\definecolor{dbluecolor}{rgb}{0.01,0.02,0.7}
\definecolor{dgreencolor}{rgb}{0.2,0.4,0.0}
\definecolor{dgraycolor}{rgb}{0.30,0.3,0.30}

\newcommand{\indep}{\perp \!\!\! \perp}

\newlength\tindent
\definecolor{delim}{RGB}{20,105,176}
\definecolor{numb}{RGB}{106, 109, 32}
\definecolor{string}{rgb}{0.64,0.08,0.08}

\newtheorem{theorem}{Theorem}
\newtheorem{proposition}[theorem]{Proposition}
\newtheorem{lemma}[theorem]{Lemma}
\newtheorem{corollary}[theorem]{Corollary}

\theoremstyle{definition}

\newtheorem{remark}[theorem]{Remark}
\newtheorem{conjecture}[theorem]{Conjecture}
\newtheorem{example}[theorem]{Example}

\newcommand\independent{\protect\mathpalette{\protect\independenT}{\perp}}
\def\independenT#1#2{\mathrel{\rlap{$#1#2$}\mkern2mu{#1#2}}}

\def\P{{\mathbb{P}}}

\def\n{{\mathbf{n}}}
\def\C{{\mathcal{C}}}

\def\Nxn{N_{X,\mathbf{n}}}

\def\proj2{\mathbb{P}^2(\mathbb{K})}

\def\Mn{\mathcal{M}_{\mathbf{n}}}

\def\n{\mathbf{n}}

\def\V{\mathcal{V}_X}
\def\Vc{\mathcal{V}_{X,\mathcal{C}}}
\def\Mc{\mathcal{M}_\mathcal{C}}
\def\Mglobal{\mathcal{M}_{\text{global}(G)}}
\def\F{\mathcal{F}}
\def\TVc{ \widetilde{\mathcal{V}}_{X,\mathcal{C}}}

\def\para{\vspace{1.5mm}}

\def\C{{\mathbb C}}

\def\Z{\mathbb{Z}}

\usepackage{mathrsfs}

\usepackage{algorithm}
\usepackage{algorithmic}

\def\proj2{\mathbb{P}^2(\mathbb{K})}

\def\bbV{\mathbb{V}}
\def\O{\mathcal{O}}

\title{Game Theory of Undirected Graphical Models}

\author{Irem Portakal}
\address{Max Planck Institute for Mathematics in the Sciences}
\email{mail@irem-portakal.de}

\author{Javier Sendra--Arranz}
\address{Max Planck Institute for Mathematics in the Sciences and University of Tübingen} \email{sendra@math.uni-tuebingen.de}

\date{\today}
\begin{document}
\maketitle
\begin{abstract}
An $n$-player game $X$ in normal form can be modeled via undirected discrete graphical models where the discrete random variables represent the players and their state spaces are the set of pure strategies. There exists an edge between the vertices of the graphical model whenever there is a dependency between the associated players. We study the Spohn conditional independence (CI) variety $\mathcal{V}_{X,\mathcal{C}}$, which is the intersection of the independence model $\Mc$ with the Spohn variety of the game $X$.  We prove a conjecture by the first author and Sturmfels that $\mathcal{V}_{X,\mathcal{C}}$ is of codimension $n$ in $\Mc$ for a generic game $X$ with binary choices. We show that the set of totally mixed CI equilibria i.e.\ the restriction of the Spohn CI variety to the open probability simplex is a smooth semialgebraic manifold for a generic game $X$ with binary choices. If the undirected graph is a disjoint union of cliques, we analyze certain algebro-geometric features of Spohn CI varieties and prove affine universality theorems.
\end{abstract}

\section{Introduction}

Game theory is an area that has historically benefited greatly from external ideas. One of the most known examples is the application of the Kakutani fixed-point theorem from topology to show the existence of Nash equilibria \cite{nash50}. Beyond topology, nonlinear algebra has also played an important role in advancing game theory. For instance, one can compute Nash equilibria by studying systems of multilinear equations. This leads to finding upper bounds for the number of totally mixed Nash equilibria of generic games which uses mixed volumes of polytopes and the BKK theorem \cite{MCKELVEY1997411}, \cite[Chapter 6]{CBMS}. \\

More recently, the concept of {\emph{correlated equilibria}}, a generalization of Nash equilibria introduced by Aumann \cite{aumann1}, was studied via the use of oriented matroids and convex geometry \cite{BHP22}. Spohn introduced yet another generalization of Nash equilibria, known as {\emph{dependency equilibria}}, by discussing how decisions made under individual rationality may differ from decisions made under collective rationality \cite{homo}.
This discussion is detailed in the classical example of the prisoner's dilemma, where the only Nash and correlated equilibrium is that both prisoners defect. 
In the concept of dependency equilibrium, the causal structure of decision situations ascends to a reflexive standpoint. This suggests that the player takes into account not only outside factors but also their own decisions and potential future decisions in the overall causal understanding of their situation.
Reflexive decision theory \cite{Spo12} is employed to rationalize the cooperation of the prisoners resulting in a dependency equilibrium. 

\para 

Every Nash equilibrium lies on the {\it Spohn variety}, i.e.\ the algebraic model of the dependency equilibria. In particular, for generic games, every Nash equilibrium is a dependency equilibrium \cite{PW24}. The algebro-geometric examination of dependency equilibrium
presented a novel perspective on understanding Nash and dependency equilibrium within the framework of {\it undirected discrete graphical models} from algebraic statistics for the first time in \cite{BI,IJ}, albeit limited to specific cases. The preference for undirected graphical models aligns with the principles of reflexive decision theory. The promise of nonlinear algebra offering a new way to expand game theory is moreover supported by the universality theorems for Nash equilibria and {\it Spohn conditional independence varieties} \cite{datta,IJ}. This paper offers a more concise exploration of general undirected graphical models and also strives to make the content inviting to both game theorists and nonlinear algebraists, notwithstanding the non-trivial nature of this objective. \\

Graphical models are widely used to build complicated dependency structures between random variables. One of the early developers of the axioms for conditional independence statements is Spohn \cite{Spo80}, who, quite coincidentally (or not), introduced dependency equilibria. We model a $(d_1 \times \cdots \times d_n)$-player game $X$ in normal form as an undirected graphical model $G = ([n], E)$ where the discrete random variables $\mathcal{X}_1, \ldots , \mathcal{X}_n$ represent the players of the game $X$, their state spaces $[d_1], \ldots, [d_n]$ represent the set of pure strategies of each player. An edge between two random variables represents the dependency of their actions for those players. A {\it conditional independent statement} has the form that $A$ is conditionally independent of $B$ given $C$ and written as $\mathcal{X}_A \indep \mathcal{X}_B \ | \ \mathcal{X}_C$, where $A,B,C$ are disjoint vertex subsets of $[n]$. This can be considered as the players in the group $A$ are conditionally independent of those in $B$ given the group of players $C$. For instance, consider a $3$-player scenario where Alice is studying for a game theory exam, Bob is assisting her with studying, and Carol is ensuring she has a good breakfast before the exam.  In the resulting undirected graphical model, which forms a line graph on three vertices, the structure illustrates that the actions of Bob are independent of the actions of Carol given Alice's. \\

We consider global Markov properties $\mathcal{C}:=\text{global}(G)$, a certain set of conditional independence statements and formally introduce the discrete conditional independence model $\mathcal{M}_{\mathcal{C}}$, Spohn CI variety $\Vc$ and CI equilibria in Section~\ref{sec: algebraic game theory}. Nash and dependency equilibria fall into two extremes of the spectrum of Spohn CI variety with the graphical model on $n$ isolated vertices and the complete graphical model on $n$ vertices, respectively. For graphs with at least one edge, the set of CI equilibria is generically a semialgebraic set of positive dimension. This leads to a much more complicated geometry than in the Nash case where we expect a finite number of points. One of the main tools for studying sophisticated semialgebraic sets is through the algebro-geometric properties of its algebraic closure. Such techniques have been proven to be extremely beneficial in fields such as optimization, convex geometry, and algebraic statistics (see e.g.\ \cite{NlA,BPT}). Following this strategy, we study the features of these semialgebraic sets i.e.\  CI equilibria, through the algebro-geometric analysis of Spohn CI varieties.\\

We focus on the games with binary choices and in Theorem~\ref{theo:conj decom}, we prove the dimension part of \cite[Conjecture 6.3]{BI}. The dimension of the Spohn CI variety can be also determined directly by the graphical model $G$ by counting the number of positive dimensional faces of the associated simplicial complex of the cliques (Corollary~\ref{cor: dimensional independence model}). We also prove that the set of totally mixed CI equilibria of an undirected graph $G$ for generic games with binary choices is either empty or a smooth semialgebraic manifold (Theorem~\ref{theo:smooth}). While the focus on binary choices enables us to prove similar universality theorems as in \cite{datta,IJ}, the study of Spohn CI varieties for games with choices beyond binary is yet to be undertaken, providing many open questions. In Section~\ref{sec: filtration of Spohn CI}, we study the filtration of Spohn CI varieties with respect to the poset of graphs on $n$ vertices. Among other examples, we also present a $4$-player game in detail where CI equilibria Pareto improve Nash equilibria in Example~\ref{ex:nash variety}.\\

Section~\ref{sec: Nash CI} is a rigorous algebro-geometric study of Nash CI variety $\Nxn$ i.e.\ Spohn CI variety where the graphical model consists of the disjoint union of $k$ cliques on ${\bf{n}}:=(n_1, \ldots, n_k)$ vertices. The independence model $\Mn$ is a Segre variety and Nash CI variety $\Nxn$ is a complete intersection in $\Mn$. These varieties can be thought as a generalization of Nash CI curve \cite{IJ} where all the cliques are isolated vertices except one is a clique on 2 vertices. One of the related approaches that makes these varieties worth studying is multi-agent reinforcement learning \cite{LWTHAM17} and partially observable Markov decision processes (POMDPs) \cite{MM22}. We study the degree of Nash CI varieties and prove that they are connected. In particular, in case of smooth Nash CI surfaces we prove that they are of general type. Lastly in Section~\ref{sec: affine universality}, in the same spirit of Datta's universality theorem for Nash equilibria and the affine universality theorems for Nash CI curves, we prove affine universality theorems for Nash CI varieties $\Nxn$ where the graphical model consists of isolated vertices and cliques of size~2.

\section{Algebraic game theory preliminaries}\label{sec: algebraic game theory}

Let $X$ be an $n$-player game. For $i\in[n]$, the $i$th player can select from $[d_i]$ strategies and the associated payoff table is a tensor $X^{(i)}$ of format $d_1 \times \cdots \times d_n$ with real entries. The entry $X^{(i)}_{j_1 \ldots j_n} \in \mathbb{R}$ represents the payoff of player $i$, when player $1$ chooses strategy $j_1$, player $2$ chooses strategy $j_2$, etc. Let $V =\mathbb{R}^{d_1} \times \cdots \times \mathbb{R}^{d_n}$ be the real vector space of all tensors, and $\mathbb{P}(V)$ the corresponding projective space. The coordinate $p_{j_1\ldots j_n}$ of $\P(V)$ is the probability that the first player chooses the strategy $j_1$, the second player $j_2$, etc. We focus on the case of totally mixed equilibria points, i.e.\ positive real points of $\P(V)$ in the open probability simplex $\Delta:=\Delta_{d_1 \ldots d_{n}-1}^{\circ}$ of dimension $d_1 \cdots d_n -1$. 

\subsection{Spohn variety and dependency equilibria}
The \textit{expected payoff} of the $i$th player is given by:
\[
PX^{(i)} = \displaystyle\sum_{j_1=1}^{d_1}\cdots \sum_{j_n=1}^{d_n}X^{(i)}_{j_1\cdots j_n}p_{j_1\cdots j_n}.
\]

Similarly, we define the \textit{conditional expected payoff} of the $i$th player as the expected payoff conditioned on player $i$ having fixed pure strategy $k \in [d_i]$ as follows 
\[
\displaystyle\sum_{j_1=1}^{d_1}\cdots\widehat{\sum_{j_i=1}^{d_i}}\cdots \sum_{j_n=1}^{d_n}X^{(i)}_{j_1\cdots k \cdots  j_n}\frac{p_{j_1\cdots k \cdots j_n}}{p_{+\cdots+k+\cdots+}},
\]
where $$p_{+\cdots+k+\cdots+}=
\displaystyle\sum_{j_1=1}^{d_1}\cdots\widehat{\sum_{j_i=1}^{d_i}}\cdots \sum_{j_n=1}^{d_n}
p_{j_1\cdots k \cdots j_n}.
$$
We say that a tensor $P\in\Delta$ is a (totally mixed) \textit{dependency equilibrium} of an $n$-player game $X$ if the conditional expected payoff of each player $i$ does not depend on their strategy $k \in [d_i]$. For mixed dependency equilibria, some of the denominators $p_{+\cdots+k+\cdots+}$ might vanish and thus some additional limit definitions is proposed \cite{SRR23, PW24}. In this paper, we focus on the totally mixed equilibria notions. We can rephrase the definition of totally mixed dependendy equilibria in terms of $2 \times 2$ minors of the following $d_i \times 2$ matrices of linear forms:
 \begin{equation}\label{eq: Spohn matrices}
M_i \,=\, M_i(P) \,\,:= \,\,\,\begin{bmatrix}
\vdots & \vdots \\
\,\,p_{+\cdots + k + \cdots+}\, &\,\, \displaystyle\sum_{j_1 = 1}^{d_1} \cdots \widehat{\displaystyle\sum_{j_{i} = 1}^{d_{i}}} \cdots \displaystyle\sum_{j_n = 1}^{d_n} X^{(i)}_{j_1 \cdots  k  \cdots j_n} p_{j_1 \cdots  k  \cdots j_n}\, \\
\vdots & \vdots 
\end{bmatrix} \! .
\end{equation}
The variety $\mathcal{V}_X \subseteq \P(V)$ defined by the $2\times 2$ minors of the matrices $M_1, \ldots, M_n$ is called the \emph{Spohn variety} of the game $X$.
The dependency equilibria of the game $X$ is then the intersection $\mathcal{V}_X \cap \Delta$. By \cite[Theorem 6]{BI}, for a generic game $X$, i.e.\ for generic payoff tables $X^{(1)},\ldots, X^{(n)}$, the Spohn variety is irreducible of codimension $d_1+\cdots +d_n-n$ and degree $d_1\cdots d_n$. 
Moreover, the set of totally mixed Nash equilibria is 
the intersection
\[
\mathcal{V}_X\cap\left( \P^{d_1-1}\times\cdots\times\P^{d_n-1}\right) \cap \Delta.
\]

For an $n$-player game given by $n$ payoff tensors $X^{(i)}$, we define the canonical linear map, called the \emph{payoff map}:
\begin{align*}
\pi_X \colon \mathcal{V}_X  &\longrightarrow \mathbb{R}^n\\
P &\mapsto (P X^{(1)}, \ldots , P X^{(n)}).
\end{align*}

The \emph{payoff region} $\mathcal{P}_X := \pi_X(\mathcal{V} \cap \Delta) \subset \pi_X (\Delta) \subset \mathbb{R}^n$ is a useful tool to study Pareto optimal dependency equilibria. It is a union of oriented matroid strata in $\mathbb{R}^n$ and its algebraic boundary is a union of irreducible hypersurfaces of degree at most $\sum_{i=1}^n d_i -n +1$ \cite[Theorem 5.5]{BI}. \\

An aspect to consider for Spohn variety is that generically it is high dimensional. The set of dependency equilibria $\V\cap\Delta$ is either empty or has the same dimension as $\V$. To drop the dimension and investigate different cases of dependencies between players resulting in a new concept of equilibria, we study the intersection of $\V$ with statistical models arising from conditional independence statements. We also make use of the payoff map e.g.\ in Example~\ref{ex:nash variety} to show how these equilibria called \emph{conditional independence equilibria} Pareto improve Nash equilibria in certain games.

\subsection{Graphical models and conditional independence equilibria}\label{subsec: graphical models and CI equilibria}

Let $G = ([n], E)$ be an undirected graph and let $\mathcal{X} = (\mathcal{X}_i \ | \ i \in [n])$ be the discrete random vectors associated to $n$ players of the given game $X$ in normal form. Let $\mathcal{X}_i$ have state space $[d_i]$, equivalently the set of pure strategies of player $i$. Each edge $(i,j) \in E$ denotes the dependence between the random variables $\mathcal{X}_i$ and $\mathcal{X}_j$ i.e.\ player~$i$ and player~$j$. We consider {\emph{Markov properties}} associated to the graph $G$, that is certain conditional independence (CI) statements that must be satisfied by all random vectors $\mathcal{X}$ consistent with the graph $G$. A pair of vertices $(a,b) \in [n]$ is said to be separated by a subset $C \subset [n] \backslash\{a,b\}$ of vertices, if every path from $a$ to $b$ contains a vertex $c \in C$. Let $A, B , C \subseteq [n]$ be disjoint subsets of $[n]$. We say that $C$ separates $A$ and $B$ if $a$ and $b$ are separated by $C$ for all $a \in A$ and $b \in B$. The {\emph{global Markov property}} global($G$) associated to $G$ consists of all CI statements $\mathcal{X}_A \independent \mathcal{X}_B \ | \ \mathcal{X}_C$ for all disjoint subsets $A$, $B$ and $C$ such that $C$ separates $A$ and $B$ in $G$. The CI statement $\mathcal{X}_A \independent \mathcal{X}_B \ | \ \mathcal{X}_C$ can be interpreted as ``given $C$, $A$ is independent from $B$ and vice versa". There are also pairwise and local Markov properties where pairwise($G$)$\subseteq$ local($G$) $\subseteq$ global($G$). However, since we focus on totally mixed equilibria i.e.\ strictly positive joint probability distributions $P \in \Delta$, $P$ satisfies the {\emph{intersection axiom}}:
$$\mathcal{X}_A \independent \mathcal{X}_{B} \ | \ \mathcal{X}_{C \cup D} \text{ and } \mathcal{X}_{A} \independent \mathcal{X}_C \ | \ \mathcal{X}_{B \cup D} \implies \mathcal{X}_A \independent \mathcal{X}_{B \cup C} \ | \ \mathcal{X}_{D} $$
Thus, the pairwise, local and global Markov property are all equivalent by Pearl and Paz in \cite{PP86}.
\begin{example}\label{ex: line graph}
    Let $G = ([4],E)$ be the line graph from Figure~\ref{fig:line graph}. All CI statements for global (and also local) Markov property associated to $G$ can be deduced by the following two CI statements via conditional independence axioms
    $$\mathcal{X}_1 \independent \mathcal{X}_{\{3,4\}} \ | \ \mathcal{X}_{2} \text{ and } \mathcal{X}_{\{1,2\}} \independent \mathcal{X}_4 \ | \ \mathcal{X}_3.$$
In this case pairwise Markov property associated to $G$ consists of 
\begin{align*}
\mathcal{X}_1 \independent \mathcal{X}_{4} \ | \ \mathcal{X}_{\{2,3\}}  \\
\mathcal{X}_1 \independent \mathcal{X}_{3} \ | \ \mathcal{X}_{\{2,4\}} \\
\mathcal{X}_2 \independent \mathcal{X}_{4} \ | \ \mathcal{X}_{\{1,3\}} 
\end{align*}
If one considers the positive joint probability distributions, by intersection axiom, they imply the two CI statements of the global Markov property. Thus, the choice of the global Markov property does not affect the study of totally mixed equilibria.
\end{example}

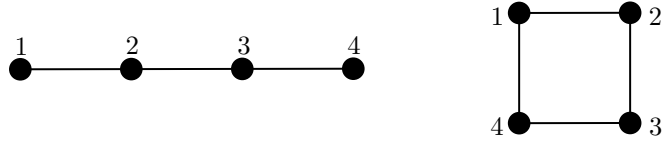
\begin{figure}
    \centering

\tikzset{every picture/.style={line width=0.75pt}}   

\begin{tikzpicture}[x=0.75pt,y=0.75pt,yscale=-0.7,xscale=0.7]

\draw  [fill={rgb, 255:red, 0; green, 0; blue, 0 }  ,fill opacity=1 ] (123.9,130.61) .. controls (123.9,126.51) and (127.23,123.18) .. (131.33,123.18) .. controls (135.44,123.18) and (138.76,126.51) .. (138.76,130.61) .. controls (138.76,134.71) and (135.44,138.04) .. (131.33,138.04) .. controls (127.23,138.04) and (123.9,134.71) .. (123.9,130.61) -- cycle ;
\draw  [fill={rgb, 255:red, 0; green, 0; blue, 0 }  ,fill opacity=1 ] (43.94,130.75) .. controls (43.94,126.65) and (47.26,123.32) .. (51.37,123.32) .. controls (55.47,123.32) and (58.79,126.65) .. (58.79,130.75) .. controls (58.79,134.85) and (55.47,138.18) .. (51.37,138.18) .. controls (47.26,138.18) and (43.94,134.85) .. (43.94,130.75) -- cycle ;
\draw    (51.37,130.75) -- (131.33,130.61) ;
\draw  [fill={rgb, 255:red, 0; green, 0; blue, 0 }  ,fill opacity=1 ] (203.87,130.47) .. controls (203.87,126.37) and (207.2,123.04) .. (211.3,123.04) .. controls (215.4,123.04) and (218.73,126.37) .. (218.73,130.47) .. controls (218.73,134.57) and (215.4,137.9) .. (211.3,137.9) .. controls (207.2,137.9) and (203.87,134.57) .. (203.87,130.47) -- cycle ;
\draw    (131.33,130.61) -- (211.3,130.47) ;
\draw    (211.3,130.47) -- (291.27,130.33) ;
\draw  [fill={rgb, 255:red, 0; green, 0; blue, 0 }  ,fill opacity=1 ] (283.84,130.33) .. controls (283.84,126.23) and (287.17,122.9) .. (291.27,122.9) .. controls (295.37,122.9) and (298.7,126.23) .. (298.7,130.33) .. controls (298.7,134.44) and (295.37,137.76) .. (291.27,137.76) .. controls (287.17,137.76) and (283.84,134.44) .. (283.84,130.33) -- cycle ;
\draw  [fill={rgb, 255:red, 0; green, 0; blue, 0 }  ,fill opacity=1 ] (403.34,90.14) .. controls (403.34,86.04) and (406.66,82.71) .. (410.77,82.71) .. controls (414.87,82.71) and (418.19,86.04) .. (418.19,90.14) .. controls (418.19,94.24) and (414.87,97.57) .. (410.77,97.57) .. controls (406.66,97.57) and (403.34,94.24) .. (403.34,90.14) -- cycle ;
\draw  [fill={rgb, 255:red, 0; green, 0; blue, 0 }  ,fill opacity=1 ] (483.31,90) .. controls (483.31,85.9) and (486.63,82.57) .. (490.73,82.57) .. controls (494.84,82.57) and (498.16,85.9) .. (498.16,90) .. controls (498.16,94.1) and (494.84,97.43) .. (490.73,97.43) .. controls (486.63,97.43) and (483.31,94.1) .. (483.31,90) -- cycle ;
\draw    (410.77,90.14) -- (490.73,90) ;
\draw  [fill={rgb, 255:red, 0; green, 0; blue, 0 }  ,fill opacity=1 ] (403.34,169.6) .. controls (403.34,165.5) and (406.66,162.17) .. (410.77,162.17) .. controls (414.87,162.17) and (418.19,165.5) .. (418.19,169.6) .. controls (418.19,173.71) and (414.87,177.03) .. (410.77,177.03) .. controls (406.66,177.03) and (403.34,173.71) .. (403.34,169.6) -- cycle ;
\draw    (410.77,169.6) -- (490.73,169.46) ;
\draw  [fill={rgb, 255:red, 0; green, 0; blue, 0 }  ,fill opacity=1 ] (483.31,169.46) .. controls (483.31,165.36) and (486.63,162.04) .. (490.73,162.04) .. controls (494.84,162.04) and (498.16,165.36) .. (498.16,169.46) .. controls (498.16,173.57) and (494.84,176.89) .. (490.73,176.89) .. controls (486.63,176.89) and (483.31,173.57) .. (483.31,169.46) -- cycle ;
\draw    (490.73,169.46) -- (490.73,90) ;
\draw    (410.77,169.6) -- (410.77,90.14) ;

\draw (45.25,105.9) node [anchor=north west][inner sep=0.75pt]  [font=\footnotesize]  {$1$};
\draw (125.25,105.15) node [anchor=north west][inner sep=0.75pt]  [font=\footnotesize]  {$2$};
\draw (205.75,105.4) node [anchor=north west][inner sep=0.75pt]  [font=\footnotesize]  {$3$};
\draw (285,105.15) node [anchor=north west][inner sep=0.75pt]  [font=\footnotesize]  {$4$};
\draw (388.19,84.54) node [anchor=north west][inner sep=0.75pt]  [font=\footnotesize]  {$1$};
\draw (502.18,84.54) node [anchor=north west][inner sep=0.75pt]  [font=\footnotesize]  {$2$};
\draw (502.19,164) node [anchor=north west][inner sep=0.75pt]  [font=\footnotesize]  {$3$};
\draw (387.88,164) node [anchor=north west][inner sep=0.75pt]  [font=\footnotesize]  {$4$};

\end{tikzpicture}

    \caption{Line graph and cycle on four vertices}
    \label{fig:line graph}
\end{figure}

\para 

For a subset of vertices representing players $A \subseteq [n]$, we let $\mathcal{R}_A : = \prod_{a \in A} [d_a]$ to be the set of pure strategy profiles for $A$. The CI statement $\mathcal{X}_A \independent \mathcal{X}_{B} \ | \ \mathcal{X}_{C}$ holds if and only if  \begin{equation}\label{eq: independence ideal} p_{i_A i_B i_C +} p_{j_A j_B i_C +} - p_{i_A j_B i_C+} p_{j_A i_B i_C +} = 0 \end{equation} for all $i_A, j_A \in \mathcal{R}_A$, $i_B, j_B \in \mathcal{R}_B$ and $i_C, j_C \in \mathcal{R}_C$ (\cite[Proposition 4.1.6]{Sul}). The notation $p_{i_A i_B i_C +}$ is the probability $P(\mathcal{X}_{A} = i_A, \mathcal{X}_{B} = i_B, \mathcal{X}_{C} = i_C)$. This means, the set of CI statements $\mathcal{C}$ := global($G$) translates into a system of quadratic polynomial equations in the entries of the joint probability distribution. 
We define the \emph{discrete conditional independence model} $\mathcal{M}_\mathcal{C} \subseteq \P(V)$ to be the projective variety defined by all probability distributions satisfying the equation (\ref{eq: independence ideal}). In the original definition from \cite[Chapter 6]{BI}, it is assumed that components lying in the hyperplanes
$\{p_{j_1j_2 \cdots j_n}=0\}$ and $\{p_{++ \cdots +} = 0\}$ have been removed from $\mathcal{M}_\mathcal{C}$, since the ultimate goal is to study the equilibria in the open probability simplex $\Delta$. We denote this union of hyperplanes by $\mathcal{W}$. The {\em Spohn conditional independence (CI) variety} is defined as:
\begin{equation*}
\label{eq:spohnCI} \mathcal{V}_{X,\mathcal{C}} \,\, := \,\,
\overline{(\mathcal{V}_X \, \cap \, \mathcal{M}_\mathcal{C})\backslash \mathcal{W}}. 
\end{equation*}
The intersection of the Spohn CI variety $\Vc$ with the open simplex $\Delta$ is called {\em totally mixed conditional independence (CI)
equilibria}. An essential observation here is that two extremes of the totally mixed CI equilibria are Nash and dependency equilibria. If one considers the graph on $n$ isolated vertices, i.e.\ no dependencies between the players, then
$\mathcal{M}_\mathcal{C} = \P^{d_1 -1} \times \cdots \times \P^{d_n - 1}$ and
CI equilibria are totally mixed Nash equilibria. On the other hand, if one considers the complete graph on $n$ vertices, then $\mathcal{M}_\mathcal{C} = \P^{d_1 \cdots d_n  - 1}$ and CI equilibria are totally mixed dependency equilibria. The central focus of this paper revolves around all the intermediate cases in between these two extremes.\\

According to the Hammersley-Clifford Theorem \cite{HC71}, we adopt an alternative definition for eliminating the special hyperplanes in both the Spohn CI variety and the independence model $\Mc$, as presented in the following proposition. Let $\mathcal{C}(G)$ be the set of all maximal cliques (complete subgraphs) of $G$.
\begin{proposition}[{\cite[Proposition 13.2.5]{Sul}}]\label{prop: parametrized model}
The parametrized discrete undirected  graphical model associated to $G$ consists of all joint probability distributions $P \in \Delta_{d_1 \cdots d_n -1}$ given by the following monomial parametrization

\begin{equation}
    p_{j_1 \cdots j_n} = \frac{1}{Z(\sigma)} \prod_{C \in \mathcal{C}(G)} \sigma_{j_C}^{(C)}
\end{equation}

where $\sigma = (\sigma^{(C)})_{C \in \mathcal{C}(G)}$ is the vector of parameters and $Z(\sigma)$ is the normalizing constant. Moreover, the positive part of the parametrized model is the hierarchical log-linear model of associated to the simplicial complex of cliques in the graph $G$.
\end{proposition}

This implies that we may consider the intersection of the positive part of parametrized graphical models with the Spohn variety in the open simplex $\Delta$ for the investigation of Spohn CI varieties and the totally mixed CI equilibria. The positive part of the parametrized toric model associated to $G$ is equal to $\Mc$ with the special hyperplanes removed. Thus, 
by Proposition~\ref{prop: parametrized model}, we may consider it as the hierarchical log-linear model associated to the simplicial complex of cliques in the graph $G$ (clique complex). We derive the dimension formula for the positive part of the parametrized (binary) model by \cite[Corollary 2.7]{HS}. Note that if $G$ is decomposable, then the parametrized discrete undirected graphical model is equal to $\mathcal{M}_{\text{global}(G)}$ without the removal of the special hyperplanes  (\cite[Theorem 4.2]{GMS06}). From now on, we focus on binary graphical models.
\begin{proposition}\label{prop: dim formula}
    Let $G = (V,E)$ be an undirected discrete binary graphical model i.e.\ $d_1 = \cdots = d_n = 2$. Then the dimension of the positive part of $\mathcal{M}_{\mathcal{C}}$ is the number of non-empty faces of the associated simplicial complex of cliques. 
\end{proposition}

\begin{example}\label{ex: chordal graph}
Consider a $4$-player game modeled with two different graphical models as in Figure~\ref{fig:line graph}. The homogenized version of the parametrization for the tree and the cycle are $$p_{j_1 j_2 j_3 j_4} = \sigma^{(12)}_{j_1 j_2 } \sigma^{(23)}_{j_2 j_3} \sigma^{(34)}_{j_3 j_4} \text{ and } p_{j_1 j_2 j_3 j_4} = \sigma^{(12)}_{j_1 j_2 } \sigma^{(23)}_{j_2 j_3} \sigma^{(34)}_{j_3 j_4} \sigma^{(14)}_{j_1 j_4} \text{ respectively}.$$  
For the line (and cycle graph), the associated simplicial complex of cliques consists of 4 cliques of size one and 3 cliques (4 cliques) of size 2. Thus, the positive part of $\Mc$ is 7-dimensional (8-dimensional).
Consider a 7-player game with binary choices modeled by the decomposable graph $G$ in Figure~\ref{fig:associated tree}. The homogenized version of the parametrization is $$p_{j_1 \cdots j_7} = \sigma^{(123)}_{j_1 j_2 j_3} \sigma^{(2345)}_{j_2 j_3 j_4 j_5} \sigma^{(2356)}_{j_2 j_3 j_5 j_6} \sigma^{(567)}_{j_5 j_6 j_7}.$$  The vanishing ideal is toric and generated by homogeneous binomials of degree 2. The dimension of the positive part of $\Mc$ is $7+13+9+2 = 31$ which is the number of non-empty faces of the associated simplicial complex of cliques. 
\end{example}

One of the main goals of this paper is to prove the conjecture on the dimension of Spohn CI varieties for generic games with binary choices, which is achieved in Theorem~\ref{theo:conj decom}. Before that, the conjecture was only proven for one-edge graphical models in \cite{IJ}.

\begin{figure}[t]
    \centering
\tikzset{every picture/.style={line width=0.75pt}}    

\begin{tikzpicture}[x=0.75pt,y=0.75pt,yscale=-1,xscale=1]

\draw  [fill={rgb, 255:red, 0; green, 0; blue, 0 }  ,fill opacity=1 ] (155,80.92) .. controls (155,78.16) and (157.24,75.92) .. (160,75.92) .. controls (162.76,75.92) and (165,78.16) .. (165,80.92) .. controls (165,83.68) and (162.76,85.92) .. (160,85.92) .. controls (157.24,85.92) and (155,83.68) .. (155,80.92) -- cycle ;
\draw  [fill={rgb, 255:red, 0; green, 0; blue, 0 }  ,fill opacity=1 ] (125,100.8) .. controls (125,98.04) and (127.24,95.8) .. (130,95.8) .. controls (132.76,95.8) and (135,98.04) .. (135,100.8) .. controls (135,103.56) and (132.76,105.8) .. (130,105.8) .. controls (127.24,105.8) and (125,103.56) .. (125,100.8) -- cycle ;
\draw    (130,100.8) -- (190,101.42) ;
\draw    (130,100.8) -- (160,80.92) ;
\draw    (160,80.92) -- (190,101.42) ;
\draw  [fill={rgb, 255:red, 0; green, 0; blue, 0 }  ,fill opacity=1 ] (185,101.42) .. controls (185,98.66) and (187.24,96.42) .. (190,96.42) .. controls (192.76,96.42) and (195,98.66) .. (195,101.42) .. controls (195,104.18) and (192.76,106.42) .. (190,106.42) .. controls (187.24,106.42) and (185,104.18) .. (185,101.42) -- cycle ;
\draw    (130,100.8) -- (160,40.92) ;
\draw    (190,101.42) -- (160,40.92) ;
\draw    (160,85.92) -- (160,40.92) ;
\draw  [fill={rgb, 255:red, 0; green, 0; blue, 0 }  ,fill opacity=1 ] (155,40.92) .. controls (155,38.16) and (157.24,35.92) .. (160,35.92) .. controls (162.76,35.92) and (165,38.16) .. (165,40.92) .. controls (165,43.68) and (162.76,45.92) .. (160,45.92) .. controls (157.24,45.92) and (155,43.68) .. (155,40.92) -- cycle ;
\draw    (190,101.42) -- (176.87,126.57) -- (160.5,160.92) ;
\draw    (130,100.8) -- (160.5,160.92) ;
\draw    (160.11,103.22) -- (160.5,160.92) ;
\draw  [fill={rgb, 255:red, 0; green, 0; blue, 0 }  ,fill opacity=1 ] (165.5,160.93) .. controls (165.44,163.69) and (163.16,165.92) .. (160.4,165.92) .. controls (157.64,165.91) and (155.44,163.66) .. (155.5,160.9) .. controls (155.56,158.14) and (157.84,155.91) .. (160.6,155.92) .. controls (163.36,155.92) and (165.56,158.17) .. (165.5,160.93) -- cycle ;
\draw    (160,80.92) -- (160.11,99) ;
\draw  [fill={rgb, 255:red, 0; green, 0; blue, 0 }  ,fill opacity=1 ] (104.6,80.8) .. controls (104.6,78.04) and (106.84,75.8) .. (109.6,75.8) .. controls (112.36,75.8) and (114.6,78.04) .. (114.6,80.8) .. controls (114.6,83.56) and (112.36,85.8) .. (109.6,85.8) .. controls (106.84,85.8) and (104.6,83.56) .. (104.6,80.8) -- cycle ;
\draw    (109.6,80.8) -- (130,100.8) ;
\draw    (109.6,80.8) -- (138,80.74) ;
\draw    (160,80.92) -- (142,80.88) ;
\draw    (190,101.42) -- (210.14,140.81) ;
\draw  [fill={rgb, 255:red, 0; green, 0; blue, 0 }  ,fill opacity=1 ] (205.14,140.81) .. controls (205.14,138.05) and (207.38,135.81) .. (210.14,135.81) .. controls (212.9,135.81) and (215.14,138.05) .. (215.14,140.81) .. controls (215.14,143.57) and (212.9,145.81) .. (210.14,145.81) .. controls (207.38,145.81) and (205.14,143.57) .. (205.14,140.81) -- cycle ;
\draw    (210.14,140.81) -- (160.5,160.92) ;
\draw  [fill={rgb, 255:red, 0; green, 0; blue, 0 }  ,fill opacity=1 ] (403.67,111.75) .. controls (403.67,108.99) and (405.91,106.75) .. (408.67,106.75) .. controls (411.43,106.75) and (413.67,108.99) .. (413.67,111.75) .. controls (413.67,114.51) and (411.43,116.75) .. (408.67,116.75) .. controls (405.91,116.75) and (403.67,114.51) .. (403.67,111.75) -- cycle ;
\draw    (408.67,111.75) -- (395.54,136.9) -- (379.17,171.25) ;
\draw  [fill={rgb, 255:red, 0; green, 0; blue, 0 }  ,fill opacity=1 ] (384.17,171.26) .. controls (384.11,174.02) and (381.83,176.26) .. (379.06,176.25) .. controls (376.3,176.24) and (374.11,174) .. (374.17,171.24) .. controls (374.22,168.48) and (376.51,166.24) .. (379.27,166.25) .. controls (382.03,166.26) and (384.22,168.5) .. (384.17,171.26) -- cycle;
\draw    (408.67,111.75) -- (428.81,151.14) ;
\draw  [fill={rgb, 255:red, 0; green, 0; blue, 0 }  ,fill opacity=1 ] (423.81,151.14) .. controls (423.81,148.38) and (426.05,146.14) .. (428.81,146.14) .. controls (431.57,146.14) and (433.81,148.38) .. (433.81,151.14) .. controls (433.81,153.9) and (431.57,156.14) .. (428.81,156.14) .. controls (426.05,156.14) and (423.81,153.9) .. (423.81,151.14) -- cycle ;
\draw    (428.81,151.14) -- (379.17,171.25) ;
\draw  [fill={rgb, 255:red, 0; green, 0; blue, 0 }  ,fill opacity=1 ] (272.33,109.13) .. controls (272.33,106.37) and (274.57,104.13) .. (277.33,104.13) .. controls (280.09,104.13) and (282.33,106.37) .. (282.33,109.13) .. controls (282.33,111.89) and (280.09,114.13) .. (277.33,114.13) .. controls (274.57,114.13) and (272.33,111.89) .. (272.33,109.13) -- cycle ;
\draw  [fill={rgb, 255:red, 0; green, 0; blue, 0 }  ,fill opacity=1 ] (251.93,89.13) .. controls (251.93,86.37) and (254.17,84.13) .. (256.93,84.13) .. controls (259.69,84.13) and (261.93,86.37) .. (261.93,89.13) .. controls (261.93,91.89) and (259.69,94.13) .. (256.93,94.13) .. controls (254.17,94.13) and (251.93,91.89) .. (251.93,89.13) -- cycle ;
\draw    (256.93,89.13) -- (277.33,109.13) ;
\draw    (256.93,89.13) -- (307.33,89.25) ;
\draw    (277.33,109.13) -- (307.33,89.25) ;
\draw  [fill={rgb, 255:red, 0; green, 0; blue, 0 }  ,fill opacity=1 ] (302.33,89.25) .. controls (302.33,86.49) and (304.57,84.25) .. (307.33,84.25) .. controls (310.09,84.25) and (312.33,86.49) .. (312.33,89.25) .. controls (312.33,92.01) and (310.09,94.25) .. (307.33,94.25) .. controls (304.57,94.25) and (302.33,92.01) .. (302.33,89.25) -- cycle ;
\draw    (321.67,81.13) -- (351.67,21.25) ;
\draw    (381.67,81.75) -- (351.67,21.25) ;
\draw    (351.67,66.25) -- (351.67,21.25) ;
\draw  [fill={rgb, 255:red, 0; green, 0; blue, 0 }  ,fill opacity=1 ] (346.67,21.25) .. controls (346.67,18.49) and (348.91,16.25) .. (351.67,16.25) .. controls (354.43,16.25) and (356.67,18.49) .. (356.67,21.25) .. controls (356.67,24.01) and (354.43,26.25) .. (351.67,26.25) .. controls (348.91,26.25) and (346.67,24.01) .. (346.67,21.25) -- cycle ;
\draw    (321.67,81.13) -- (381.67,81.75) ;
\draw    (321.67,81.13) -- (351.67,66.25) ;
\draw    (351.67,66.25) -- (381.67,81.75) ;
\draw  [fill={rgb, 255:red, 0; green, 0; blue, 0 }  ,fill opacity=1 ] (316.67,81.13) .. controls (316.67,78.37) and (318.91,76.13) .. (321.67,76.13) .. controls (324.43,76.13) and (326.67,78.37) .. (326.67,81.13) .. controls (326.67,83.89) and (324.43,86.13) .. (321.67,86.13) .. controls (318.91,86.13) and (316.67,83.89) .. (316.67,81.13) -- cycle ;
\draw  [fill={rgb, 255:red, 0; green, 0; blue, 0 }  ,fill opacity=1 ] (376.67,81.75) .. controls (376.67,78.99) and (378.91,76.75) .. (381.67,76.75) .. controls (384.43,76.75) and (386.67,78.99) .. (386.67,81.75) .. controls (386.67,84.51) and (384.43,86.75) .. (381.67,86.75) .. controls (378.91,86.75) and (376.67,84.51) .. (376.67,81.75) -- cycle ;
\draw    (381.33,111.08) -- (368.2,136.23) -- (351.83,170.58) ;
\draw    (321.33,110.47) -- (351.83,170.58) ;
\draw    (351.44,112.89) -- (351.83,170.58) ;
\draw  [fill={rgb, 255:red, 0; green, 0; blue, 0 }  ,fill opacity=1 ] (356.83,170.6) .. controls (356.78,173.36) and (354.49,175.59) .. (351.73,175.58) .. controls (348.97,175.58) and (346.78,173.33) .. (346.83,170.57) .. controls (346.89,167.81) and (349.17,165.58) .. (351.94,165.58) .. controls (354.7,165.59) and (356.89,167.84) .. (356.83,170.6) -- cycle ;
\draw    (321.33,110.47) -- (381.33,111.08) ;
\draw    (321.33,110.47) -- (351.33,95.58) ;
\draw    (351.33,95.58) -- (381.33,111.08) ;
\draw  [fill={rgb, 255:red, 0; green, 0; blue, 0 }  ,fill opacity=1 ] (316.33,110.47) .. controls (316.33,107.71) and (318.57,105.47) .. (321.33,105.47) .. controls (324.09,105.47) and (326.33,107.71) .. (326.33,110.47) .. controls (326.33,113.23) and (324.09,115.47) .. (321.33,115.47) .. controls (318.57,115.47) and (316.33,113.23) .. (316.33,110.47) -- cycle ;
\draw  [fill={rgb, 255:red, 0; green, 0; blue, 0 }  ,fill opacity=1 ] (346.33,95.58) .. controls (346.33,92.82) and (348.57,90.58) .. (351.33,90.58) .. controls (354.09,90.58) and (356.33,92.82) .. (356.33,95.58) .. controls (356.33,98.34) and (354.09,100.58) .. (351.33,100.58) .. controls (348.57,100.58) and (346.33,98.34) .. (346.33,95.58) -- cycle ;
\draw  [fill={rgb, 255:red, 0; green, 0; blue, 0 }  ,fill opacity=1 ] (376.33,111.08) .. controls (376.33,108.32) and (378.57,106.08) .. (381.33,106.08) .. controls (384.09,106.08) and (386.33,108.32) .. (386.33,111.08) .. controls (386.33,113.84) and (384.09,116.08) .. (381.33,116.08) .. controls (378.57,116.08) and (376.33,113.84) .. (376.33,111.08) -- cycle ;
\draw  [fill={rgb, 255:red, 0; green, 0; blue, 0 }  ,fill opacity=1 ] (346.67,66.25) .. controls (346.67,63.49) and (348.91,61.25) .. (351.67,61.25) .. controls (354.43,61.25) and (356.67,63.49) .. (356.67,66.25) .. controls (356.67,69.01) and (354.43,71.25) .. (351.67,71.25) .. controls (348.91,71.25) and (346.67,69.01) .. (346.67,66.25) -- cycle ;
\draw    (351.33,94.81) -- (351.67,108.5) ;
\draw (164.95,71.92) node [anchor=north west][inner sep=0.75pt]   [align=left] {{\footnotesize $3$}};
\draw (167.95,34.72) node [anchor=north west][inner sep=0.75pt]   [align=left] {{\footnotesize $4$}};
\draw (90.95,74.92) node [anchor=north west][inner sep=0.75pt]   [align=left] {{\footnotesize $1$}};
\draw (196.95,96.32) node [anchor=north west][inner sep=0.75pt]   [align=left] {{\footnotesize $5$}};
\draw (113.95,100.92) node [anchor=north west][inner sep=0.75pt]   [align=left] {{\footnotesize $2$}};
\draw (143.95,154.92) node [anchor=north west][inner sep=0.75pt]   [align=left] {{\footnotesize $6$}};
\draw (215.95,135.32) node [anchor=north west][inner sep=0.75pt]   [align=left] {{\footnotesize $7$}};
\draw (152.93,191.33) node [anchor=north west][inner sep=0.75pt]   [align=left] {{\small $G$}};
\draw (274.6,191.33) node [anchor=north west][inner sep=0.75pt]   [align=left] {{\small Maximal cliques of $G$}};

\end{tikzpicture}

    \caption{The decomposable graph $G$ has $4$ maximal cliques. Two of them have $3$ vertices and the other two have $4$ vertices.}
    \label{fig:associated tree}
\end{figure}
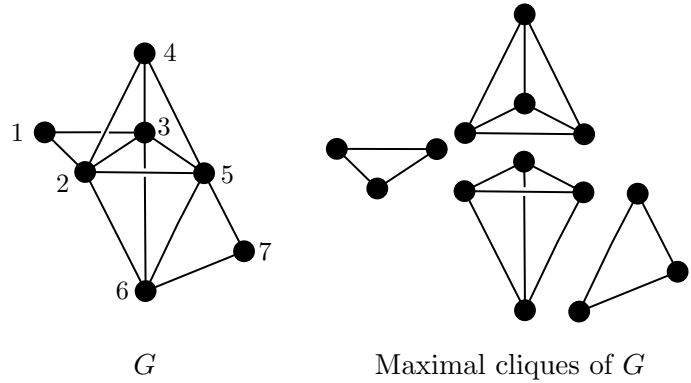

\begin{conjecture}[{\cite[Conjecture 24]{BI}}]\label{conj:dim}
Let $G$ be the undirected graphical model that is modelling a generic $n$-player game $X$ with binary choices in normal form. Let $\mathcal{C} = \text{global}(G)$ and $\Mc$ be the discrete conditional independence model of $G$. Then, the corresponding Spohn CI variety $\Vc$ has codimension $n$ in $\Mc$.
\end{conjecture}

The specification on binary choices also allowed us to prove some universality theorems in Section~\ref{sec: Nash CI} for Nash CI varieties which are Spohn CI varieties for undirected graphical models that are disjoint union of cliques. In this setting, the Spohn variety and $\Mc$ are projective subvarieties in the projective space $\P^{2^n-1}$ defined by  the determinants of the $2 \times 2$ matrices of linear forms
\begin{equation*}\label{eq: Spohn matrices 2}
M_i \,=\,\,\begin{bmatrix}

\,\,p_{+\cdots + 1 + \cdots+}\, &\,\, \displaystyle\sum_{j_1 = 1}^{d_1} \cdots \widehat{\displaystyle\sum_{j_{i} = 1}^{d_{i}}} \cdots \displaystyle\sum_{j_n = 1}^{d_n} X^{(i)}_{j_1 \cdots  1  \cdots j_n} p_{j_1 \cdots  1  \cdots j_n}\, \\
\,\,p_{+\cdots + 2 + \cdots+}\, &\,\, \displaystyle\sum_{j_1 = 1}^{d_1} \cdots \widehat{\displaystyle\sum_{j_{i} = 1}^{d_{i}}} \cdots \displaystyle\sum_{j_n = 1}^{d_n} X^{(i)}_{j_1 \cdots  2  \cdots j_n} p_{j_1 \cdots  2  \cdots j_n}\,
\end{bmatrix} \! \text{, for }i\in[n].
\end{equation*}

In particular, the Spohn CI variety $\Vc$ has codimension at most $n$ in $\Mc$.

\para

\section{Dimension of Spohn CI varieties}\label{sec: dimension of Spohn CI}

In this section, we prove Conjecture \ref{conj:dim} for any undirected graphical model. In \cite{IJ}, the conjecture is proven for one-edge Bayesian networks and equivalently one-edge undirected graphical models. In this case, the discrete conditional independence model is a Segre variety (Corollary~\ref{co:segre graph}). The parametrization of this Segre variety plays a fundamental role in the computation of the dimension of this Spohn CI variety (Section~\ref{subsec: graphical models and CI equilibria}). Let $G = ([n], E)$ be an undirected graphical model with $n$ vertices, and let $\mathcal{C}(G)$ be the set of the maximal cliques of $G$.
For a clique $C\in\mathcal{C}(G)$, we consider the torus
 \begin{equation}\label{eq:torus graph}
\mathbb{T}_C:=\left(\C^{*}\right)^{2^{|C|}} \text{ with coordinates }\sigma^{(C)}_{j_{C}} \text{ for } j_C=(j_i)_{i\in[C]}\in [2]^{|C|},
 \end{equation}
 where $[C]$ denotes the set of vertices of $C$ and $|C|$ denotes the number of vertices.
 
 By Proposition~\ref{prop: parametrized model}, we consider the homogenized parametrization of the affine cone of the independence model $\widetilde{\Mc}$ as the following map 
\[
\begin{array}{cccc}
\phi:&
\mathbb{T}:=\displaystyle\prod_{C\in\mathcal{C}(G)}\mathbb{T}_C
& \longrightarrow & \P(V),
\end{array}
\]
given by 
\begin{equation}\label{eq:para decomposable}
   p_{j_{1}\cdots j_{n}} = \displaystyle\prod_{C\in\mathcal{C}(G)}\sigma^{(C)}_{j_{C}}.
\end{equation}
  
Now, we evaluate the determinants of the matrices $M_1,\ldots,M_n$ in \eqref{eq: Spohn matrices} by \eqref{eq:para decomposable}. 
 This is the same strategy used in \cite{IJ} for computing the equations of the Nash CI curve. As in the Nash CI curve case, we distinguish two cases depending on whether the graph has isolated vertices or not.
For $i\in[n]$, let $G_i$ be the connected component of $G$ containing $i$. 
We denote the set of maximal cliques of $G_i$ by $\mathcal{C}(G)_i$, and we consider the set $N_{G}(i)$ of the vertices in $G_i$ distinct than $i$. 
In other words, $N_{G}(i)$ is the set of vertices of $G$ distinct than $i$ that are connected to the vertex $i$. The cardinal of $N_G(i)$ is denoted by $c_i$. Note that if $i$ is an isolated vertex, $N_G(i)$ is empty.
Now, for $j=(j_k)_{k\in N_G(i)}\in[2]^{c_i}$ and $a\in[2]$, we consider the index $j(a)=(j_k)_{k\in N_G(i)\cup\{i\}}\in[2]^{c_i+1}$ where $j_i= a$. Given such index and a clique $C\in \mathcal{C}(G)_i$, we also consider the index $j_C(a)=(j_k)_{k\in[C]}\in [2]^{|C|}$, where $j_i=a$. Note that a clique $C\in \mathcal{C}(G)_i$ might not contain the vertex $i$, in which case, $j_C(a)=j_C =(j_k)_{k\in[C]}$.
Using this notation, we define the monomial and the payoff entry 
\[
\mathfrak{S}^{(i)}_{j,a}:= \prod_{C\in\mathcal{C}(G)_i}\sigma^{(C)}_{j_C(a)}\]
for  $a\in[2]$ and $j=(j_k)_{k\in N_G(i)}\in[2]^{c_i}$.  
Then, the evaluation of the determinant of $M_i$ at \eqref{eq:para decomposable} is the determinant of the matrix
\begin{equation}\label{eq:det graphs}
    \begin{pmatrix}
        \displaystyle\sum_{j\in [2]^{c_i}}   \displaystyle\sum_{j'\in [2]^{n-c_i-1}}\mathfrak{S}_{j,1}^{(i)}\prod_{C\not\in\mathcal{C}(G)_i}\sigma^{(C)}_{j'_C}
        & \,\,\,
 \displaystyle\sum_{j\in [2]^{c_i}}   \displaystyle\sum_{j'\in [2]^{n-c_i-1}} X^{(i)}_{\cdots 1\cdots }\,\,\mathfrak{S}_{j,1}^{(i)}\prod_{C\not\in\mathcal{C}(G)_i}\sigma^{(C)}_{j'_C}
        \\ \noalign{\vspace*{4mm}}
          \displaystyle\sum_{j\in [2]^{c_i}}   \displaystyle\sum_{j'\in [2]^{n-c_i-1}} \mathfrak{S}_{j,2}^{(i)}\prod_{C\not\in\mathcal{C}(G)_i}\sigma^{(C)}_{j'_C}
        & \,\,\,  \displaystyle\sum_{j\in [2]^{c_i}}   \displaystyle\sum_{j'\in [2]^{n-c_i-1}}  X^{(i)}_{\cdots 2\cdots }\,\,\mathfrak{S}_{j,2}^{(i)}\prod_{C\not\in\mathcal{C}(G)_i}\sigma^{(C)}_{j'_C}
        
    \end{pmatrix}.
\end{equation}
By $X^{(i)}_{\cdots a \cdots}$, we mean the payoff entries that correspond to the parametrization $$\mathfrak{S}_{j,a}^{(i)}\prod_{C\not\in\mathcal{C}(G)_i}\sigma^{(C)}_{j'_C}$$ on each term of the sum. From the first column of \eqref{eq:det graphs}, we deduce that the determinant of \eqref{eq:det graphs} is the product of 
\begin{equation}    \label{eq:factor 1}
    \displaystyle\sum_{j\in [2]^{n-c_i-1}}
    \prod_{C\not\in\mathcal{C}(G)_i}\sigma^{(C)}_{j_C}
\end{equation}
and the polynomial 
\begin{equation}\label{eq: F gen graphs}
\mathrm{det}   \begin{pmatrix}
        \displaystyle\sum_{j\in [2]^{c_i}} \mathfrak{S}_{j,1}^{(i)}
        & \,\,\,
  \displaystyle\sum_{j\in [2]^{c_i}}   \displaystyle\sum_{j'\in [2]^{n-c_i-1}} X^{(i)}_{\cdots 1\cdots}\,\,\mathfrak{S}_{j,1}^{(i)}\prod_{C\not\in\mathcal{C}(G)_i}\sigma^{(C)}_{j'_C}
        \\ \noalign{\vspace*{4mm}} 
         \displaystyle\sum_{j\in [2]^{c_i}} \mathfrak{S}_{j,2}^{(i)}
        &  \,\,\,  \displaystyle\sum_{j\in [2]^{c_i}}   \displaystyle\sum_{j'\in [2]^{n-c_i-1}} X^{(i)}_{\cdots 2\cdots }\,\,\mathfrak{S}_{j,2}^{(i)}\prod_{C\not\in\mathcal{C}(G)_i}\sigma^{(C)}_{j'_C}
        
    \end{pmatrix}
\end{equation}
We define the polynomial $F_i$ as the determinant \eqref{eq: F gen graphs}. Note that if $G_i = G$, then (\ref{eq:factor 1}) is 1. A similar factorization can also be observed in Proposition~\ref{prop:prod graphs}.
Assume now that $i$ is an isolated vertex of the graph. By abuse of notation, we also denote the maximal clique defined by this isolated vertex by $i$. In this case, the determinant \eqref{eq: F gen graphs} is
\[
\mathrm{det}   \begin{pmatrix}
        \displaystyle\sigma_{1}^{(i)}
        &\,\,\, 
 \displaystyle\sum_{j\in [2]^{n-1}} X^{(i)}_{j_1\cdots 1\cdots j_n}\sigma_{1}^{(i)}\prod_{C\in\mathcal{C}(G)\setminus\{i\}}\sigma^{(C)}_{j_C}
        \\ \noalign{\vspace*{4mm}} 
         \displaystyle\sigma_{2}^{(i)}
        & \,\,\, \displaystyle\sum_{j\in [2]^{n-1}}  X^{(i)}_{j_1\cdots 2\cdots j_n}\sigma_{2}^{(i)}\prod_{C\in\mathcal{C}(G)\setminus\{i\}}\sigma^{(C)}_{j_C}
    \end{pmatrix}.
\]
We obtain that the above determinant is the product of $\sigma_1^{(i)}\sigma_2^{(2)}$ and the determinant
\begin{equation}
    \label{eq:det iso vertes graphs}
    \mathrm{det}   \begin{pmatrix}
        \displaystyle 1
        & \,\,\,
 \displaystyle\sum_{j\in [2]^{n-1}} X^{(i)}_{j_1\cdots 1\cdots j_n}\prod_{C\in\mathcal{C}(G)\setminus\{i\}}\sigma^{(C)}_{j_C}
        \\ \noalign{\vspace*{4mm}} 
         \displaystyle 1
        &  \,\,\, \displaystyle\sum_{j\in [2]^{n-1}}  X^{(i)}_{j_1\cdots 2\cdots j_n}\prod_{C\in\mathcal{C}(G)\setminus\{i\}}\sigma^{(C)}_{j_C}
        
    \end{pmatrix}.
\end{equation}
For an isolated vertex $i$ of the graph, we define the polynomial $F_i$ as the determinant \eqref{eq:det iso vertes graphs}.
We denote the variety defined by $F_1,\ldots,F_n$ in $\mathbb{T}$ by $Y_X$. By construction $Y_X$ is contained in $\phi^{-1}(\Mc\cap\V)$. 

\para 

\begin{lemma}
    \label{lemma:spohn contained}
    For any $X$, $\phi^{-1}(\Vc)$ is contained in $Y_X$.
\end{lemma}
\begin{proof}
 To construct the polynomials $F_1,\ldots,F_n$ we have removed some factors of the determinant \eqref{eq:det graphs} for when G is not connected. The image via $\phi$ of the varieties defined by each of these factors
are contained in some of the hyperplanes $\{p_{j_1j_2 \cdots j_n}=0\}$ and $\{p_{++ \cdots +} = 0\}$. Assume that the factor \eqref{eq:factor 1} vanishes. By \eqref{eq:para decomposable} we get that 
\[
p_{+\cdots+1+\cdots+}=
 \displaystyle\left(\sum_{j\in [2]^{n-c_i-1}}
    \prod_{C\not\in\mathcal{C}(G)_i}\sigma^{(C)}_{j_C}\right)
    \left(
\sum_{j\in [2]^{c_i}}
    \prod_{C\in\mathcal{C}(G)_i}\sigma^{(C)}_{j_C(1)}
    \right)=0.
\]
Therefore, $Y_X$ is obtained by removing some components from $\phi^{-1}(\Mc\cap\V)$ contained in the preimage via $\phi$ of the hyperplanes $\{p_{j_1j_2 \cdots j_n}=0\}$ and $\{p_{++ \cdots +} = 0\}$. 
We deduce that the preimage of the Spohn CI variety $\Vc$ through $\phi$ is contained in $Y_X$. 
\end{proof}

\para 

Our strategy is to analyse the dimension of $Y_X$ to compute the dimension of $\Vc$. To do so, we analyse the base loci of the linear systems defined by   $F_1,\ldots,F_n$. 

\para 

Note that $F_i$ is a multihomogeneous polynomial in the coordinates of $\mathbb{T}$. Its multidegree depends on whether $i$ is an isolated vertex or not. 
Assume that $i\in[n]$ is an isolated vertex. Then,
for $C\in\mathcal{C}(G)$, the degree of $F_i$ in the coordinates of $\mathbb{T}_C$ 
is $0$ if $C=i$ and $1$ otherwise. In other words, the multidegree of $F_i$ is given by the integer vector where $e_C$ is the canonical basis element:
\[
\sum_{C\in\mathcal{C}(G)\setminus\{i\}}e_C. 
\]
Assume now that $i$ is not an isolated vertex.
The multidegree of $F_i$ is given by the integer vector
\[
\displaystyle\sum_{C\not\in\mathcal{C}(G)_i}e_C+\sum_{C\in\mathcal{C}(G)_i}2e_C.
\]
We denote the space of multihomogeneous polynomials in the coordinates of $\mathbb{T}$, of the same multidegree as $F_i$, 
by $V_i$. 
In particular, $F_i$ is contained in $V_i$ for any game $X$.
For $i\in[n]$ we consider the linear map 
\[
\begin{array}{ccc}
\mathbb{R}^{2^n}&\longrightarrow&V_i\\
X^{(i)}&\longmapsto&F_i.
\end{array}
\]
We denote the image of this map by $\Lambda_i$. We use Bertini's Theorem (see \cite[Theorem 8.18]{Hart}) to compute the dimension of $Y_X$. To apply this strategy, we analyse the base locus of $\Lambda_i$.
First, if $i$ is an isolated vertex, as in \cite[Section 4.1]{IJ} one obtains that $\Lambda_i=V_i$. Now, assume that $i$ is not an isolated vertex.
Then, $F_i$ can be written as a linear combination of polynomials that are the product of a determinant of the form
\begin{equation}\label{eq:factors of F_i}
\mathrm{det}   \begin{pmatrix}
        \displaystyle\sum_{j\in [2]^{c_i}} \mathfrak{S}_{j,1}^{(i)}
        & \,\,\,
 \displaystyle\sum_{j\in [2]^{c_i}} Y^{(i)}_{j(1)}\mathfrak{S}_{j,1}^{(i)}
        \\  \noalign{\vspace*{4mm}}
         \displaystyle\sum_{j\in [2]^{c_i}} \mathfrak{S}_{j,2}^{(i)}
        & \,\,\, \displaystyle\sum_{j\in [2]^{c_i}}  Y^{(i)}_{j(2)}\mathfrak{S}_{j,2}^{(i)}
        
    \end{pmatrix},
\end{equation}
for some $Y^{(i)}_{j(1)},Y^{(i)}_{j(2)}\in\mathbb{R}$,
and a multihomogeneous polynomial $L$ of multidegree
\begin{equation}\label{eq:multdegree}
\sum_{C\not\in\mathcal{C}(G)_i}e_C.
\end{equation}
 Moreover, for any polynomial that is the product of $L$ and \eqref{eq:factors of F_i}, there exists a game $X$ such that $F_i$ equals this product. We denote the vector space of all multihomogeneous polynomials of the form \eqref{eq:factors of F_i} by $W_i$. Then, $\Lambda_i$ is the tensor product of $W_i$ and the complete linear system of multihomogeneous polynomials with multidegree \eqref{eq:multdegree}. In particular,
 $\Lambda_i$ and $W_i$ have the same base locus.

\para 

\begin{lemma}\label{lemma:gens lin syst}
    For $i\in[n]$ not being an isolated vertex, the linear system $W_i$ is generated by the polynomials
    \begin{enumerate}
        \item For $a\in[2]^{c_i}$, $\mathfrak{S}^{(i)}_{a,1}\left(\displaystyle\sum_{j\in[2]^{c_i}} \mathfrak{S}^{(i)}_{j,2}
        \right)$.
        \item For $a\in[2]^{c_i}$, $\mathfrak{S}^{(i)}_{a,2}\left(\displaystyle\sum_{j\in[2]^{c_i}} \mathfrak{S}^{(i)}_{j,1}
        \right)$.
        \item $
         \mathfrak{S}^{(i)}_{\mathbbm{1},1} \mathfrak{S}^{(i)}_{\mathbbm{1},2}-
         \displaystyle\sum_{j,k\in[2]^{c_i}\setminus\{\mathbbm{1}\}} \mathfrak{S}^{(i)}_{j,1}\mathfrak{S}^{(i)}_{k,2}
        $, where $\mathbbm{1}=(1,\ldots,1)\in [2]^{c_i}$.
    \end{enumerate}
\end{lemma}
\begin{proof}
We write the determinant \eqref{eq:factors of F_i} as
\begin{equation}\label{eq:det exp}
\displaystyle\sum_{j,k\in [2]^{c_i}}  A^{(i)}_{j,k}\mathfrak{S}_{k,1}^{(i)}\mathfrak{S}_{j,2}^{(i)}, 
\end{equation}
where 
\[
A^{(i)}_{j,k} = Y^{(i)}_{j(2)}-Y^{(i)}_{k(1)}.
\]
Note that for $j,k\in [2]^{c_i}$, we have that 
\[
A^{(i)}_{j,k} -A^{(i)}_{j,\mathbbm{1}}-A^{(i)}_{\mathbbm{1},k}+A^{(i)}_{\mathbbm{1},\mathbbm{1}}=  Y^{(i)}_{j(2)}-Y^{(i)}_{k(1)}
- Y^{(i)}_{j(2)}+Y^{(i)}_{\mathbbm{1}(1)}
- Y^{(i)}_{\mathbbm{1}(2)}+Y^{(i)}_{k(1)}
+ Y^{(i)}_{\mathbbm{1}(2)}-Y^{(i)}_{\mathbbm{1}(1)}=0,
\]
and we deduce that $$A^{(i)}_{j,k} =A^{(i)}_{j,\mathbbm{1}}+A^{(i)}_{\mathbbm{1},k}-A^{(i)}_{\mathbbm{1},\mathbbm{1}}\,\,\,\text{for $j,k\neq\mathbbm{1}$}.$$ Therefore, we can write the polynomial \eqref{eq:det exp} as 
\[
\displaystyle
\sum_{j\in [2]^{c_i}}  A^{(i)}_{j,\mathbbm{1}}\mathfrak{S}_{\mathbbm{1},1}^{(i)}\mathfrak{S}_{j,2}^{(i)}+
\sum_{j\in [2]^{c_i}\setminus\{\mathbbm{1}\}}  A^{(i)}_{\mathbbm{1},j}\mathfrak{S}_{j,1}^{(i)}\mathfrak{S}_{\mathbbm{1},2}^{(i)}+
\sum_{j,k\in [2]^{c_i}\setminus\{\mathbbm{1}\}}  (A^{(i)}_{j,\mathbbm{1}}+A^{(i)}_{\mathbbm{1},k}-A^{(i)}_{\mathbbm{1},\mathbbm{1}})\mathfrak{S}_{k,1}^{(i)}\mathfrak{S}_{j,2}^{(i)}.
\]
The proof follows by fixing in the above expression all the coefficients $A^{(i)}_{j,\mathbbm{1}},A^{(i)}_{\mathbbm{1},j}$ except one to be zero.
\end{proof}

\para 

Once we have computed the generators of $W_i$, we deal with the computation of their base loci.

\para 

\begin{lemma}
    \label{lemma:base locus graphs}

    For $i\in[n]$ not being an isolated vertex, the base locus of $W_i$ is  
    \begin{equation}\label{eq:base loci}
    \bbV(G_1,G_2)\cup \bbV(\mathfrak{S}_{a,1}^{(i)}:a\in[2]^{c_i})\cup\bbV(\mathfrak{S}_{a,2}^{(i)}:a\in[2]^{c_i})
    \end{equation}
    where 
    \[
    G_1 = \displaystyle\sum_{j\in[2]^{c_i}} \mathfrak{S}^{(i)}_{j,2}\text{ and } G_2 = \displaystyle\sum_{j\in[2]^{c_i}} \mathfrak{S}^{(i)}_{j,1}
    \]
\end{lemma}
\begin{proof}
Let $Z_1,Z_2,Z_3$ be the three varieties in the union \eqref{eq:base loci} respectively, and 
let $Z$ be the variety defined by the ideal generated by all the polynomials listed in Lemma \ref{lemma:gens lin syst}. We show that $Z_1\cup Z_2\cup Z_3=Z$. First, note that the first row of the matrix in \eqref{eq:factors of F_i} vanishes at $Z_2$. In particular, the determinant \eqref{eq:factors of F_i} vanishes at $Z_2$, and hence, $Z_2$ is contained in $Z$. Similarly, $Z_3$ is contained in $Z$. Now, the first column of the matrix \eqref{eq:factors of F_i} vanishes at $Z_1$. Therefore $Z_1$ is also contained in $Z$.

\para 

Next, we assume that $p$ is a point in $Z$ not contained in $Z_1$. Then, either $G_1(p)\neq 0$ or $G_2(p)\neq 0$. Assume that $G_1(p)$ does not vanish.
Note that the  first type of polynomials in Lemma \ref{lemma:gens lin syst} are of the form $\mathfrak{S}_{a,1}^{(i)}G_1$ for $a\in[2]^{c_i}$. Since $G_1(p)\neq 0$, we deduce that $\mathfrak{S}_{a,1}^{(i)}$ vanishes at $p$ for $a\in[2]^{c_i}$. Therefore, $p$ is contained in $ Z_2$. 
Similarly, if $G_2(p)\neq 0$, then $p\in Z_3$.
We conclude that  $Z=Z_1\cup Z_2\cup Z_3$.    
\end{proof}

\para 

Lemma \ref{lemma:base locus graphs} allows us to prove Conjecture \ref{conj:dim}.

\para

\begin{theorem}\label{theo:conj decom}
Conjecture \ref{conj:dim} holds for any undirected graphical model.
\end{theorem}
\begin{proof}
Let $\Vc$ be the Spohn CI variety of a generic game $X$ and let $\TVc$ be the preimage of $\Vc$ through the monomial map $\phi$ in \eqref{eq:para decomposable}. 
By Lemma \ref{lemma:spohn contained}, $\TVc$ is contained in $Y_X$. 
Let $B_X$ be the intersection of $Y_X$ and the union of the base loci of $\Lambda_1,\ldots, \Lambda_n$, and let $\widetilde{Y}_X$ be the Zariski closure of $Y_X\setminus B_X$ in $\mathbb{T}$. Recall that the base locus of $\Lambda_i$ is either empty if $i$ is an isolated vertex, or it is given by Lemma \ref{lemma:base locus graphs}.
By Bertini's Theorem (see \cite[Theorem 8.18]{Hart}), we get that $Y_X\setminus B_X$ and $\widetilde{Y}_X$ have codimension $n$ in $\mathbb{T}$ for a generic game $X$. Now, note that for $i\in[n]$, the image of the base locus of $\Lambda_i$ via $\phi$ is contained in the union of the hyperplanes $\{p_{j_1j_2 \cdots j_n}=0\}$ and $\{p_{++ \cdots +} = 0\}$. This implies that $\TVc$ is contained in $\widetilde{Y}_X$, and we deduce that 
\[
n= \mathrm{codim}_{\mathbb{T}} \tilde{Y}
\leq
\mathrm{codim}_{\mathbb{T}}\TVc .
\]
  Now, the proof follows from the fact that $\mathrm{codim}_{\mathbb{T}}\TVc\leq \mathrm{codim}_{\Mc}\Vc$
and  that $\mathrm{codim}_{\Mc}\Vc\leq n$.
\end{proof}

\para 

We deduce the following result as a consequence of Proposition~\ref{prop: dim formula} and Theorem~\ref{theo:conj decom}.

\para 

\begin{corollary}\label{cor: dimensional independence model}
Let $G=([n],E)$ be a discrete undirected binary graphical model. Then, for generic payoff tables,  the dimension of the Spohn CI variety $\Vc$ is the number of positive dimensional faces of the associated simplicial complex of the cliques. In other words, for generic payoff tables, the dimension of $\Vc$ is the number of cliques of $G$ with at least two vertices.
\end{corollary}

    \para 

\begin{example}
For generic payoff tables, if $G = ([n], E)$ is a line graph or a cycle (see Figure \ref{fig:line graph}), the dimension of the Spohn CI variety is determined by counting the number of edges. This is because the clique complex consists exclusively of one-dimensional simplices. 
On the other hand in the case of the decomposable graph from Figure~\ref{fig:associated tree}, the Spohn CI variety is $31-7 = 24$ dimensional.
\end{example}

Now, we use the above analysis of the base locus the linear systems $\Lambda_i$ to study the smoothness of generic Spohn CI varieties.

\begin{theorem}\label{theo:smooth}
    For a generic $n$--player binary game $X$, the Spohn CI variety $\mathcal{V}_{X,\mathcal{C}}$ is smooth away from the hyperplanes $\{p_{j_1j_2 \cdots j_n}=0\}$ and $\{p_{++ \cdots +} = 0\}$. In particular, the set of totally mixed CI equilibria of $X$ is a smooth semialgebraic manifold. 
\end{theorem}
\begin{proof}
Let $\Vc$ be the Spohn CI variety of a generic game $X$ and let $\TVc$ be the preimage of $\Vc$ through the monomial map $\phi$ in \eqref{eq:para decomposable}. By construction, $Y_X$ and $\TVc$ coincide away from the preimage of  the hyperplanes $\{p_{j_1j_2 \cdots j_n}=0\}$ and $\{p_{++ \cdots +} = 0\}$ through $\phi$. Recall that the base locus of the linear systems $\Gamma_i$ is contained in the preimage of these hyperplanes. Hence, by Bertini's Theorem (see \cite[Theorem 8.18]{Hart}) we deduce that $Y_X$ and $\TVc$ are smooth away from the preimage of these hyperplanes. Using \cite[\href{https://stacks.math.columbia.edu/tag/02KL}{Tag 02KL}]{Stack}, we deduce that $\Vc$ is smooth away from the hyperplanes $\{p_{j_1j_2 \cdots j_n}=0\}$ and $\{p_{++ \cdots +} = 0\}$. Finally, these hyperplanes only intersect the probability simplex in its boundary. Hence, we conclude that the set of totally mixed CI equilibria is a smooth semialgebraic manifold.
\end{proof}

\begin{remark}
    Using \cite[Theorem 2.2.9]{mangolte} and Theorem~\ref{theo:conj decom} and Theorem~\ref{theo:smooth}, we have that for generic binary games the set of totally mixed CI equilibria of an undirected graph $G$ is either empty or a smooth manifold whose dimension equals the number of cliques of $G$ with at least two vertices.
\end{remark}

\section{Filtration of Spohn CI varieties}\label{sec: filtration of Spohn CI}

We explore how Spohn CI varieties of different undirected graphs are related.
Let \[ \text{$G = ([n], E(G))$ and $G' = ([n], E(G'))$} \] be two undirected graphs. We say that $G$ is a subgraph of $G'$, denoted by $G\subseteq G'$, if $E(G) \subseteq E(G')$.

\begin{lemma}\label{lemma:inc graphs}
    Let $G\subseteq G'$ and let $\Vc$ and $\Vc'$ be the Spohn CI variety of $G$ and $G'$ respectively. Then, $\Vc$ is a subvariety of $\Vc'$. The analogous inclusion holds for the corresponding sets of totally mixed CI equilibria.
\end{lemma}
 \begin{proof}  
Let $G\subseteq G'$, 
then, we have that  $\mathrm{global}(G')$ is contained in $\mathrm{global}(G)$. This implies that 
$I_{\mathrm{global}(G')}\subseteq I_{\mathrm{global}(G)}$, and hence, $\mathcal{M}_{\mathrm{global}(G)}$ is a subvariety of $\mathcal{M}_{\mathrm{global}(G')}$. In particular, we deduce that the Spohn CI variety, corresponding to $G$, is contained in the Spohn CI variety corresponding to $G'$. 
 \end{proof}

Let $G$ be the complete graph with $n$ vertices. The inclusion of graphs gives to the set of subgraphs of $G$, with $n$ vertices, a structure of poset. By Lemma \ref{lemma:inc graphs}, we get a poset structure on the set of Spohn CI varieties (similarly with totally mixed CI equilibria). In this poset the initial and terminal objects are the set of totally mixed Nash equilibria and the set of totally mixed dependency equilibria. In other words, 
 the set of totally mixed CI equilibria always contains the set of totally mixed Nash equilibria and it is always contained in the set of totally mixed dependency equilibria.

\para 

\begin{example}\label{ex:filtration}
    For $n=3$ we have $8$ subgraphs of the complete graph on $3$ vertices: one with no edges, $3$ with one edge, $3$ with two edges, and the complete graph. The poset structure of the set, formed by these $8$ graphs, is shown in Figure \ref{fig:poset Graph}. 
    In particular, we get a similar picture for the corresponding independence varieties and Spohn CI varieties. A {\tt Macaulay2}  computation shows that the dimension of the independence varieties for a graph with $0$, $1$, $2$ or $3$ edges is $3$, $4$, $5$ and $7$ respectively. Therefore, by Theorem \ref{theo:conj decom},   the dimension of the corresponding Spohn CI varieties are $0$, $1$, $2$, and $4$ respectively. This shows that, in the poset of Spohn CI varieties, there might be dimensional gaps. 
\end{example}

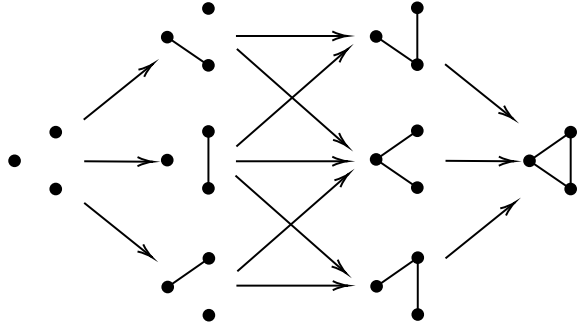
\begin{figure}[h]
    \centering

\tikzset{every picture/.style={line width=0.75pt}}

\begin{tikzpicture}[x=0.75pt,y=0.75pt,yscale=-0.7,xscale=0.7]

\draw  [fill={rgb, 255:red, 0; green, 0; blue, 0 }  ,fill opacity=1 ] (66.67,140.17) .. controls (66.67,138.05) and (68.38,136.33) .. (70.5,136.33) .. controls (72.62,136.33) and (74.33,138.05) .. (74.33,140.17) .. controls (74.33,142.28) and (72.62,144) .. (70.5,144) .. controls (68.38,144) and (66.67,142.28) .. (66.67,140.17) -- cycle ;
\draw  [fill={rgb, 255:red, 0; green, 0; blue, 0 }  ,fill opacity=1 ] (96.33,160.5) .. controls (96.33,158.38) and (98.05,156.67) .. (100.17,156.67) .. controls (102.28,156.67) and (104,158.38) .. (104,160.5) .. controls (104,162.62) and (102.28,164.33) .. (100.17,164.33) .. controls (98.05,164.33) and (96.33,162.62) .. (96.33,160.5) -- cycle ;
\draw  [fill={rgb, 255:red, 0; green, 0; blue, 0 }  ,fill opacity=1 ] (96.33,119.5) .. controls (96.33,117.38) and (98.05,115.67) .. (100.17,115.67) .. controls (102.28,115.67) and (104,117.38) .. (104,119.5) .. controls (104,121.62) and (102.28,123.33) .. (100.17,123.33) .. controls (98.05,123.33) and (96.33,121.62) .. (96.33,119.5) -- cycle ;
\draw  [fill={rgb, 255:red, 0; green, 0; blue, 0 }  ,fill opacity=1 ] (176.67,139.67) .. controls (176.67,137.55) and (178.38,135.83) .. (180.5,135.83) .. controls (182.62,135.83) and (184.33,137.55) .. (184.33,139.67) .. controls (184.33,141.78) and (182.62,143.5) .. (180.5,143.5) .. controls (178.38,143.5) and (176.67,141.78) .. (176.67,139.67) -- cycle ;
\draw  [fill={rgb, 255:red, 0; green, 0; blue, 0 }  ,fill opacity=1 ] (206.33,160) .. controls (206.33,157.88) and (208.05,156.17) .. (210.17,156.17) .. controls (212.28,156.17) and (214,157.88) .. (214,160) .. controls (214,162.12) and (212.28,163.83) .. (210.17,163.83) .. controls (208.05,163.83) and (206.33,162.12) .. (206.33,160) -- cycle ;
\draw  [fill={rgb, 255:red, 0; green, 0; blue, 0 }  ,fill opacity=1 ] (206.33,119) .. controls (206.33,116.88) and (208.05,115.17) .. (210.17,115.17) .. controls (212.28,115.17) and (214,116.88) .. (214,119) .. controls (214,121.12) and (212.28,122.83) .. (210.17,122.83) .. controls (208.05,122.83) and (206.33,121.12) .. (206.33,119) -- cycle ;
\draw  [fill={rgb, 255:red, 0; green, 0; blue, 0 }  ,fill opacity=1 ] (176.67,51) .. controls (176.67,48.88) and (178.38,47.17) .. (180.5,47.17) .. controls (182.62,47.17) and (184.33,48.88) .. (184.33,51) .. controls (184.33,53.12) and (182.62,54.83) .. (180.5,54.83) .. controls (178.38,54.83) and (176.67,53.12) .. (176.67,51) -- cycle ;
\draw  [fill={rgb, 255:red, 0; green, 0; blue, 0 }  ,fill opacity=1 ] (206.33,71.33) .. controls (206.33,69.22) and (208.05,67.5) .. (210.17,67.5) .. controls (212.28,67.5) and (214,69.22) .. (214,71.33) .. controls (214,73.45) and (212.28,75.17) .. (210.17,75.17) .. controls (208.05,75.17) and (206.33,73.45) .. (206.33,71.33) -- cycle ;
\draw  [fill={rgb, 255:red, 0; green, 0; blue, 0 }  ,fill opacity=1 ] (206.33,30.33) .. controls (206.33,28.22) and (208.05,26.5) .. (210.17,26.5) .. controls (212.28,26.5) and (214,28.22) .. (214,30.33) .. controls (214,32.45) and (212.28,34.17) .. (210.17,34.17) .. controls (208.05,34.17) and (206.33,32.45) .. (206.33,30.33) -- cycle ;
\draw  [fill={rgb, 255:red, 0; green, 0; blue, 0 }  ,fill opacity=1 ] (177,231.17) .. controls (177,229.05) and (178.72,227.33) .. (180.83,227.33) .. controls (182.95,227.33) and (184.67,229.05) .. (184.67,231.17) .. controls (184.67,233.28) and (182.95,235) .. (180.83,235) .. controls (178.72,235) and (177,233.28) .. (177,231.17) -- cycle ;
\draw  [fill={rgb, 255:red, 0; green, 0; blue, 0 }  ,fill opacity=1 ] (206.67,251.5) .. controls (206.67,249.38) and (208.38,247.67) .. (210.5,247.67) .. controls (212.62,247.67) and (214.33,249.38) .. (214.33,251.5) .. controls (214.33,253.62) and (212.62,255.33) .. (210.5,255.33) .. controls (208.38,255.33) and (206.67,253.62) .. (206.67,251.5) -- cycle ;
\draw  [fill={rgb, 255:red, 0; green, 0; blue, 0 }  ,fill opacity=1 ] (206.67,210.5) .. controls (206.67,208.38) and (208.38,206.67) .. (210.5,206.67) .. controls (212.62,206.67) and (214.33,208.38) .. (214.33,210.5) .. controls (214.33,212.62) and (212.62,214.33) .. (210.5,214.33) .. controls (208.38,214.33) and (206.67,212.62) .. (206.67,210.5) -- cycle ;
\draw  [fill={rgb, 255:red, 0; green, 0; blue, 0 }  ,fill opacity=1 ] (327.17,139.17) .. controls (327.17,137.05) and (328.88,135.33) .. (331,135.33) .. controls (333.12,135.33) and (334.83,137.05) .. (334.83,139.17) .. controls (334.83,141.28) and (333.12,143) .. (331,143) .. controls (328.88,143) and (327.17,141.28) .. (327.17,139.17) -- cycle ;
\draw  [fill={rgb, 255:red, 0; green, 0; blue, 0 }  ,fill opacity=1 ] (356.83,159.5) .. controls (356.83,157.38) and (358.55,155.67) .. (360.67,155.67) .. controls (362.78,155.67) and (364.5,157.38) .. (364.5,159.5) .. controls (364.5,161.62) and (362.78,163.33) .. (360.67,163.33) .. controls (358.55,163.33) and (356.83,161.62) .. (356.83,159.5) -- cycle ;
\draw  [fill={rgb, 255:red, 0; green, 0; blue, 0 }  ,fill opacity=1 ] (356.83,118.5) .. controls (356.83,116.38) and (358.55,114.67) .. (360.67,114.67) .. controls (362.78,114.67) and (364.5,116.38) .. (364.5,118.5) .. controls (364.5,120.62) and (362.78,122.33) .. (360.67,122.33) .. controls (358.55,122.33) and (356.83,120.62) .. (356.83,118.5) -- cycle ;
\draw  [fill={rgb, 255:red, 0; green, 0; blue, 0 }  ,fill opacity=1 ] (327.17,50.5) .. controls (327.17,48.38) and (328.88,46.67) .. (331,46.67) .. controls (333.12,46.67) and (334.83,48.38) .. (334.83,50.5) .. controls (334.83,52.62) and (333.12,54.33) .. (331,54.33) .. controls (328.88,54.33) and (327.17,52.62) .. (327.17,50.5) -- cycle ;
\draw  [fill={rgb, 255:red, 0; green, 0; blue, 0 }  ,fill opacity=1 ] (356.83,70.83) .. controls (356.83,68.72) and (358.55,67) .. (360.67,67) .. controls (362.78,67) and (364.5,68.72) .. (364.5,70.83) .. controls (364.5,72.95) and (362.78,74.67) .. (360.67,74.67) .. controls (358.55,74.67) and (356.83,72.95) .. (356.83,70.83) -- cycle ;
\draw  [fill={rgb, 255:red, 0; green, 0; blue, 0 }  ,fill opacity=1 ] (356.83,29.83) .. controls (356.83,27.72) and (358.55,26) .. (360.67,26) .. controls (362.78,26) and (364.5,27.72) .. (364.5,29.83) .. controls (364.5,31.95) and (362.78,33.67) .. (360.67,33.67) .. controls (358.55,33.67) and (356.83,31.95) .. (356.83,29.83) -- cycle ;
\draw  [fill={rgb, 255:red, 0; green, 0; blue, 0 }  ,fill opacity=1 ] (327.5,230.67) .. controls (327.5,228.55) and (329.22,226.83) .. (331.33,226.83) .. controls (333.45,226.83) and (335.17,228.55) .. (335.17,230.67) .. controls (335.17,232.78) and (333.45,234.5) .. (331.33,234.5) .. controls (329.22,234.5) and (327.5,232.78) .. (327.5,230.67) -- cycle ;
\draw  [fill={rgb, 255:red, 0; green, 0; blue, 0 }  ,fill opacity=1 ] (357.17,251) .. controls (357.17,248.88) and (358.88,247.17) .. (361,247.17) .. controls (363.12,247.17) and (364.83,248.88) .. (364.83,251) .. controls (364.83,253.12) and (363.12,254.83) .. (361,254.83) .. controls (358.88,254.83) and (357.17,253.12) .. (357.17,251) -- cycle ;
\draw  [fill={rgb, 255:red, 0; green, 0; blue, 0 }  ,fill opacity=1 ] (357.17,210) .. controls (357.17,207.88) and (358.88,206.17) .. (361,206.17) .. controls (363.12,206.17) and (364.83,207.88) .. (364.83,210) .. controls (364.83,212.12) and (363.12,213.83) .. (361,213.83) .. controls (358.88,213.83) and (357.17,212.12) .. (357.17,210) -- cycle ;
\draw  [fill={rgb, 255:red, 0; green, 0; blue, 0 }  ,fill opacity=1 ] (437.67,140.17) .. controls (437.67,138.05) and (439.38,136.33) .. (441.5,136.33) .. controls (443.62,136.33) and (445.33,138.05) .. (445.33,140.17) .. controls (445.33,142.28) and (443.62,144) .. (441.5,144) .. controls (439.38,144) and (437.67,142.28) .. (437.67,140.17) -- cycle ;
\draw  [fill={rgb, 255:red, 0; green, 0; blue, 0 }  ,fill opacity=1 ] (467.33,160.5) .. controls (467.33,158.38) and (469.05,156.67) .. (471.17,156.67) .. controls (473.28,156.67) and (475,158.38) .. (475,160.5) .. controls (475,162.62) and (473.28,164.33) .. (471.17,164.33) .. controls (469.05,164.33) and (467.33,162.62) .. (467.33,160.5) -- cycle ;
\draw  [fill={rgb, 255:red, 0; green, 0; blue, 0 }  ,fill opacity=1 ] (467.33,119.5) .. controls (467.33,117.38) and (469.05,115.67) .. (471.17,115.67) .. controls (473.28,115.67) and (475,117.38) .. (475,119.5) .. controls (475,121.62) and (473.28,123.33) .. (471.17,123.33) .. controls (469.05,123.33) and (467.33,121.62) .. (467.33,119.5) -- cycle ;
\draw    (210.17,119) -- (210.17,160) ;
\draw    (360.67,29.83) -- (360.67,70.83) ;
\draw    (471.17,119.5) -- (471.17,160.5) ;
\draw    (361,210) -- (361,251) ;
\draw    (180.5,51) -- (210.17,71.33) ;
\draw    (331,50.5) -- (360.67,70.83) ;
\draw    (441.5,140.17) -- (471.17,160.5) ;
\draw    (331,139.17) -- (360.67,159.5) ;
\draw    (331.33,230.67) -- (361,210) ;
\draw    (441.5,140.17) -- (471.17,119.5) ;
\draw    (180.83,231.17) -- (210.5,210.5) ;
\draw    (120.33,110.5) -- (168.78,71.42) ;
\draw [shift={(170.33,70.17)}, rotate = 141.11] [color={rgb, 255:red, 0; green, 0; blue, 0 }  ][line width=0.75]    (10.93,-3.29) .. controls (6.95,-1.4) and (3.31,-0.3) .. (0,0) .. controls (3.31,0.3) and (6.95,1.4) .. (10.93,3.29)   ;
\draw    (381,210.5) -- (429.44,171.42) ;
\draw [shift={(431,170.17)}, rotate = 141.11] [color={rgb, 255:red, 0; green, 0; blue, 0 }  ][line width=0.75]    (10.93,-3.29) .. controls (6.95,-1.4) and (3.31,-0.3) .. (0,0) .. controls (3.31,0.3) and (6.95,1.4) .. (10.93,3.29)   ;
\draw    (120.67,140.5) -- (168,140.82) ;
\draw [shift={(170,140.83)}, rotate = 180.39] [color={rgb, 255:red, 0; green, 0; blue, 0 }  ][line width=0.75]    (10.93,-3.29) .. controls (6.95,-1.4) and (3.31,-0.3) .. (0,0) .. controls (3.31,0.3) and (6.95,1.4) .. (10.93,3.29)   ;
\draw    (381,140.17) -- (428.33,140.49) ;
\draw [shift={(430.33,140.5)}, rotate = 180.39] [color={rgb, 255:red, 0; green, 0; blue, 0 }  ][line width=0.75]    (10.93,-3.29) .. controls (6.95,-1.4) and (3.31,-0.3) .. (0,0) .. controls (3.31,0.3) and (6.95,1.4) .. (10.93,3.29)   ;
\draw    (230,140.5) -- (308.33,140.5) ;
\draw [shift={(310.33,140.5)}, rotate = 180] [color={rgb, 255:red, 0; green, 0; blue, 0 }  ][line width=0.75]    (10.93,-3.29) .. controls (6.95,-1.4) and (3.31,-0.3) .. (0,0) .. controls (3.31,0.3) and (6.95,1.4) .. (10.93,3.29)   ;
\draw    (230.33,230.17) -- (308.67,230.17) ;
\draw [shift={(310.67,230.17)}, rotate = 180] [color={rgb, 255:red, 0; green, 0; blue, 0 }  ][line width=0.75]    (10.93,-3.29) .. controls (6.95,-1.4) and (3.31,-0.3) .. (0,0) .. controls (3.31,0.3) and (6.95,1.4) .. (10.93,3.29)   ;
\draw    (230,50.17) -- (308.33,50.17) ;
\draw [shift={(310.33,50.17)}, rotate = 180] [color={rgb, 255:red, 0; green, 0; blue, 0 }  ][line width=0.75]    (10.93,-3.29) .. controls (6.95,-1.4) and (3.31,-0.3) .. (0,0) .. controls (3.31,0.3) and (6.95,1.4) .. (10.93,3.29)   ;
\draw    (230.67,59.5) -- (308.83,128.51) ;
\draw [shift={(310.33,129.83)}, rotate = 221.44] [color={rgb, 255:red, 0; green, 0; blue, 0 }  ][line width=0.75]    (10.93,-3.29) .. controls (6.95,-1.4) and (3.31,-0.3) .. (0,0) .. controls (3.31,0.3) and (6.95,1.4) .. (10.93,3.29)   ;
\draw    (229.67,150.83) -- (307.83,219.84) ;
\draw [shift={(309.33,221.17)}, rotate = 221.44] [color={rgb, 255:red, 0; green, 0; blue, 0 }  ][line width=0.75]    (10.93,-3.29) .. controls (6.95,-1.4) and (3.31,-0.3) .. (0,0) .. controls (3.31,0.3) and (6.95,1.4) .. (10.93,3.29)   ;
\draw    (231,219.83) -- (309.83,150.82) ;
\draw [shift={(311.33,149.5)}, rotate = 138.8] [color={rgb, 255:red, 0; green, 0; blue, 0 }  ][line width=0.75]    (10.93,-3.29) .. controls (6.95,-1.4) and (3.31,-0.3) .. (0,0) .. controls (3.31,0.3) and (6.95,1.4) .. (10.93,3.29)   ;
\draw    (230.33,129.83) -- (309.16,60.82) ;
\draw [shift={(310.67,59.5)}, rotate = 138.8] [color={rgb, 255:red, 0; green, 0; blue, 0 }  ][line width=0.75]    (10.93,-3.29) .. controls (6.95,-1.4) and (3.31,-0.3) .. (0,0) .. controls (3.31,0.3) and (6.95,1.4) .. (10.93,3.29)   ;
\draw    (120.67,170.5) -- (168.11,208.58) ;
\draw [shift={(169.67,209.83)}, rotate = 218.75] [color={rgb, 255:red, 0; green, 0; blue, 0 }  ][line width=0.75]    (10.93,-3.29) .. controls (6.95,-1.4) and (3.31,-0.3) .. (0,0) .. controls (3.31,0.3) and (6.95,1.4) .. (10.93,3.29)   ;
\draw    (380.67,70.5) -- (428.11,108.58) ;
\draw [shift={(429.67,109.83)}, rotate = 218.75] [color={rgb, 255:red, 0; green, 0; blue, 0 }  ][line width=0.75]    (10.93,-3.29) .. controls (6.95,-1.4) and (3.31,-0.3) .. (0,0) .. controls (3.31,0.3) and (6.95,1.4) .. (10.93,3.29)   ;
\draw    (331,139.17) -- (360.67,118.5) ;
\end{tikzpicture}

    \caption{The poset structure of the set of subgraphs of the complete graph on $3$ vertices with respect to the inclusion.}
    \label{fig:poset Graph}
\end{figure}

\begin{example}\label{ex:nash variety}
In the following, we construct a $(2 \times 2 \times 2 \times 2)$-game $X$ to present detailed computations of CI equilibria. The study of the payoff region of this game shows that there are totally mixed CI equilibria that Pareto improves the totally mixed Nash equilibria. We set the payoff tables whose nonzero entries are as follows:
\[\begin{array}{c}
X^{(1)}_{1111} =X^{(2)}_{1112} =
X^{(3)}_{1111} =X^{(4)}_{1211} =1, 
\,\,\,
 X^{(2)}_{1121} =
X^{(4)}_{2111} =-10,\,\,\,X^{(2)}_{2221}  =
X^{(4)}_{2122}  = -16,
\\
X^{(1)}_{2111}  =X^{(2)}_{1212} =
X^{(3)}_{1121}  =X^{(4)}_{1212} =3, 
\,\,\,
X^{(2)}_{1221}  =
X^{(4)}_{2112} =-14,\,\,\,
 X^{(2)}_{2121} =
X^{(4)}_{2121}  =-12,

\\
X^{(1)}_{1211}  =X^{(2)}_{2112} =
X^{(3)}_{1112} =X^{(4)}_{1221} =X^{(1)}_{2122}  =
X^{(3)}_{2221} =2, 
\\
X^{(1)}_{2211} =X^{(2)}_{2212} =
X^{(3)}_{1122}  =X^{(4)}_{1222} = X^{(2)}_{1222} =
X^{(4)}_{2212}  =4.

\end{array}.
\]

Let $G_i$ be an undirected graphical model from Figure~\ref{fig:filtration graph} and $\mathcal{C}_i = \text{global}(G_i)$ for $i \in [4]$. The graphical model $G_4$ is the disjoint union of two cliques and thus the independence model $\mathcal{M}_{\mathcal{C}_4}$ is the Segre variety $\P^3\times\P^3$. The Spohn CI variety $\mathcal{V}_{X, \mathcal{C}_4}$ is a subvariety of $\P^{15}$ lying in the the intersection of $\mathcal{M}_{\mathcal{C}_4}$ and the Spohn variety $\mathcal{V}_X$. 
Let $\sigma_{ij}^{(1)}$ and $\sigma_{ij}^{(2)}$ be the coordinates of the first and second $\P^3$ factor of $\mathcal{M}_{\mathcal{C}_4}$.
As a subvariety of $\mathcal{M}_{\mathcal{C}_4}$, the Spohn CI variety $\mathcal{V}_{X, \mathcal{C}_4}$ is defined by the following four polynomials that are products of a linear form $l_i$ and a quadratic form $q_i$ for $i \in [4]$:
\begin{equation}\label{eq:nash surface}
\begin{array}{cc}
l_1 q_1 : = (\sigma^{(2)}_{11}-2\sigma^{(2)}_{22})(2\sigma^{(1)}_{11}\sigma^{(1)}_{21}+\sigma^{(1)}_{21}\sigma^{(1)}_{12}+3\sigma^{(1)}_{11}\sigma^{(1)}_{22}+2\sigma^{(1)}_{12}\sigma^{(1)}_{22}),
\\
l_2 q_2 := (\sigma^{(2)}_{21}-2\sigma^{(2)}_{12})(2\sigma^{(1)}_{11}\sigma^{(1)}_{12}+\sigma^{(1)}_{21}\sigma^{(1)}_{12}+3\sigma^{(1)}_{11}\sigma^{(1)}_{22}+2\sigma^{(1)}_{21}\sigma^{(1)}_{22}),
\\
l_3 q_3:=(\sigma^{(1)}_{11}-2\sigma^{(1)}_{22})(2\sigma^{(2)}_{11}\sigma^{(2)}_{21}+\sigma^{(2)}_{21}\sigma^{(2)}_{12}+3\sigma^{(2)}_{11}\sigma^{(2)}_{22}+2\sigma^{(2)}_{12}\sigma^{(2)}_{22}),
\\
l_4 q_4:= (\sigma^{(1)}_{21}-2\sigma^{(1)}_{12})(2\sigma^{(2)}_{11}\sigma^{(2)}_{12}+\sigma^{(2)}_{21}\sigma^{(2)}_{12}+3\sigma^{(2)}_{11}\sigma^{(2)}_{22}+2\sigma^{(2)}_{21}\sigma^{(2)}_{22}).
\end{array}
\end{equation}
In Section \ref{sec: Nash CI} we provide general formulas for these equations.
The Spohn CI variety $\mathcal{V}_{X, \mathcal{C}_4}$  is the union of $14$
complete intersection surfaces in $\P^3\times \P^3$. One of them is the zero locus of the ideal  $\langle l_1,l_2,l_3,l_4 \rangle$, which is isomorphic to $\P^1\times\P^1$. Four of these surfaces are the zero locus of ideals of the form $\langle l_{i_1},l_{i_2},l_{i_3},q_{i_4}\rangle$ for $\{i_1,i_2,i_3,i_4\}=[4]$ distinct. Each of these surfaces are isomorphic to the disjoint union of two planes. Similarly, we get six varieties that are the zero locus of ideals of the form $\langle l_{i_1},l_{i_2},q_{i_3},q_{i_4}\rangle$ for $\{i_1,i_2,i_3,i_4\}=[4]$ distinct. The varieties defined by the ideals $\langle l_1,l_2,q_3,q_4\rangle$ and $\langle q_1,q_2,l_3,l_4\rangle$ are empty, whereas the other $4$ ideals define surfaces that are the product of $2$ smooth conics.
We get four surfaces defined by ideals of the form $\langle l_{i_1},q_{i_2},q_{i_3},q_{i_4}\rangle$ for $\{i_1,i_2,i_3,i_4\}=[4]$ distinct. Each of these surfaces are the disjoint union of $4$ quadric surfaces in $\P^3$.
    Finally, the ideal $\langle q_1,q_2,q_3,q_4\rangle$ leads to a surface which is the product of two degree $4$ curves in $\P^3$.
    We conclude that $\mathcal{V}_{X,C_4}$ is the union of $14$ complete intersection surfaces but it has $30$ irreducible components.
\para 

The quadratic surface defined by $q_i$ in the corresponding $\P^3$ does not intersect the open simplex $\Delta:=\Delta_{15}^{\circ}$ for $i \in [4]$. In particular, we deduce that the only component of $\mathcal{V}_{X, C_4}$ intersecting $\Delta$ is $\mathcal{L}:=\mathbb{V}( l_1,l_2,l_3,l_4)$. Therefore, the set of totally mixed CI equilibria of $G_4$ is $\mathcal{L} \cap(\Delta_3^{\circ}\times\Delta_3^{\circ})$ where $\Delta_3^\circ$ is the open simplex in $\P^3$. Equivalently, as a subset of $\P^{15}$, the set of totally mixed  CI equilibria is
    \[
    \left\{p\in\P^{15}:\begin{array}{cc}
    p_{1111}=2p_{1122}=2p_{2211}=4p_{2222},&
    p_{1212}=2p_{1221}=2p_{2112}=4p_{2121}\\
    p_{1112}=2p_{1121}=2p_{2212}=4p_{2221},&
    p_{1211}=2p_{2111}=2p_{1222}=4p_{2122}\\
     p_{2222}p_{2121}=p_{2221}p_{2122},&
    p_{2222},p_{2121},p_{2221},p_{2122}>0
    \end{array}
    \right\}.
    \]
Note that the set of totally mixed CI equilibria is contained in a $3$-dimensional projective space defined by the $12$ linear equations in the previous expression. We identify this projective space with $\P^3$. Let $z_0,z_1,z_2,z_3$ be the coordinates of $\P^3$ corresponding to $p_{2222},p_{2221},p_{2122},p_{2121}$. 
We may view $\mathcal{L}$ as the surface $\bbV(z_0z_3-z_1z_2)\subset\P^3$ and  the set of totally mixed CI equilibria as the intersection of $\bbV(z_0z_3-z_1z_2)$ with the open simplex $\Delta_3^{\circ}$. In Figure~\ref{fig:filtration graph}, we illustrate the poset of subgraphs of $G_4$ and similarly a poset of inclusions of Spohn CI varieties.
The {\color{rgb, 255:red, 95; green, 154; blue, 131}Segre surface} contains two components of {\color{rgb, 255:red, 0; green, 252; blue, 245} Nash CI curves} and in their intersection lies the \textbf{set of totally mixed Nash equilibria}: The only components of the Nash CI curves $\mathcal{V}_{X,\mathcal{C}_2}$ and $\mathcal{V}_{X,\mathcal{C}_3}$ intersecting the open simplex are the line $L_1=\mathbb{V}(z_0-z_1,z_2-z_3)$ and $L_2=\mathbb{V}(z_0-z_2,z_1-z_3)$. The intersection of $L_1$ and $L_2$ is the unique totally mixed Nash equilibria which is the point $p=[1,1,1,1]$.
Note that $\mathcal{L}$ is a ruled surface and through each point $q$ in $\mathcal{L}$ there are exactly two lines contained in $\mathcal{L}$ passing through $q$. In our case, the totally mixed Nash equilibria is the point $p$ in $\mathcal{L}$ and the set of totally mixed CI equilibria of the graphs $G_1$ and $G_2$ correspond to the two lines in $S$ passing through $p$ respectively. This is illustrated in Figure \ref{fig:filtration graph}.

\para

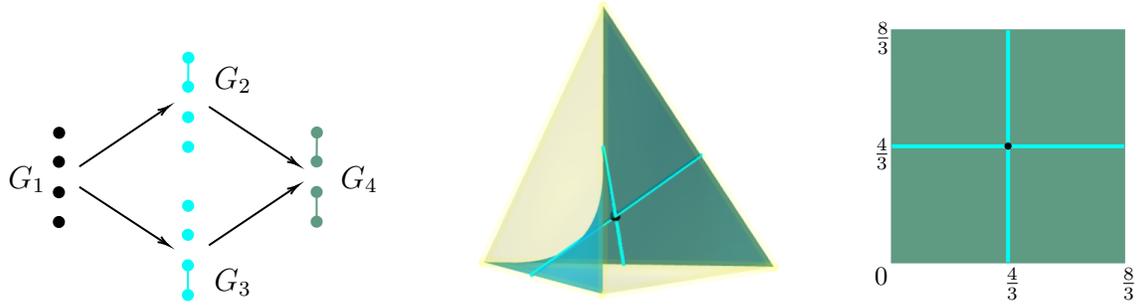
\begin{figure}
    \centering

\tikzset{every picture/.style={line width=0.75pt}}

\begin{tikzpicture}[x=0.75pt,y=0.75pt,yscale=-0.5,xscale=0.5]

\draw  [color={rgb, 255:red, 0; green, 252; blue, 245 }  ,draw opacity=1 ][fill={rgb, 255:red, 0; green, 252; blue, 245 }  ,fill opacity=1 ] (325.01,20.5) .. controls (325.01,17.69) and (327.29,15.41) .. (330.1,15.41) .. controls (332.91,15.41) and (335.19,17.69) .. (335.19,20.5) .. controls (335.19,23.31) and (332.91,25.59) .. (330.1,25.59) .. controls (327.29,25.59) and (325.01,23.31) .. (325.01,20.5) -- cycle ;
\draw  [color={rgb, 255:red, 0; green, 252; blue, 245 }  ,draw opacity=1 ][fill={rgb, 255:red, 0; green, 252; blue, 245 }  ,fill opacity=1 ] (324.76,49.75) .. controls (324.76,46.94) and (327.04,44.66) .. (329.85,44.66) .. controls (332.66,44.66) and (334.94,46.94) .. (334.94,49.75) .. controls (334.94,52.56) and (332.66,54.84) .. (329.85,54.84) .. controls (327.04,54.84) and (324.76,52.56) .. (324.76,49.75) -- cycle ;
\draw  [color={rgb, 255:red, 0; green, 252; blue, 245 }  ,draw opacity=1 ][fill={rgb, 255:red, 0; green, 252; blue, 245 }  ,fill opacity=1 ] (324.76,80.53) .. controls (324.76,77.72) and (327.04,75.44) .. (329.85,75.44) .. controls (332.66,75.44) and (334.94,77.72) .. (334.94,80.53) .. controls (334.94,83.34) and (332.66,85.62) .. (329.85,85.62) .. controls (327.04,85.62) and (324.76,83.34) .. (324.76,80.53) -- cycle ;
\draw  [color={rgb, 255:red, 0; green, 252; blue, 245 }  ,draw opacity=1 ][fill={rgb, 255:red, 0; green, 252; blue, 245 }  ,fill opacity=1 ] (324.76,110.53) .. controls (324.76,107.72) and (327.04,105.44) .. (329.85,105.44) .. controls (332.66,105.44) and (334.94,107.72) .. (334.94,110.53) .. controls (334.94,113.34) and (332.66,115.62) .. (329.85,115.62) .. controls (327.04,115.62) and (324.76,113.34) .. (324.76,110.53) -- cycle ;
\draw [color={rgb, 255:red, 0; green, 252; blue, 245 }  ,draw opacity=1 ][fill={rgb, 255:red, 0; green, 252; blue, 245 }  ,fill opacity=1 ]   (330.1,20.5) -- (329.85,49.75) ;
\draw  [color={rgb, 255:red, 0; green, 252; blue, 245 }  ,draw opacity=1 ][fill={rgb, 255:red, 0; green, 252; blue, 245 }  ,fill opacity=1 ] (324.76,170) .. controls (324.76,167.19) and (327.04,164.91) .. (329.85,164.91) .. controls (332.66,164.91) and (334.94,167.19) .. (334.94,170) .. controls (334.94,172.81) and (332.66,175.09) .. (329.85,175.09) .. controls (327.04,175.09) and (324.76,172.81) .. (324.76,170) -- cycle ;
\draw  [color={rgb, 255:red, 0; green, 252; blue, 245 }  ,draw opacity=1 ][fill={rgb, 255:red, 0; green, 252; blue, 245 }  ,fill opacity=1 ] (324.51,199.25) .. controls (324.51,196.44) and (326.79,194.16) .. (329.6,194.16) .. controls (332.41,194.16) and (334.69,196.44) .. (334.69,199.25) .. controls (334.69,202.06) and (332.41,204.34) .. (329.6,204.34) .. controls (326.79,204.34) and (324.51,202.06) .. (324.51,199.25) -- cycle ;
\draw  [color={rgb, 255:red, 0; green, 252; blue, 245 }  ,draw opacity=1 ][fill={rgb, 255:red, 0; green, 252; blue, 245 }  ,fill opacity=1 ] (324.51,230.03) .. controls (324.51,227.22) and (326.79,224.94) .. (329.6,224.94) .. controls (332.41,224.94) and (334.69,227.22) .. (334.69,230.03) .. controls (334.69,232.84) and (332.41,235.12) .. (329.6,235.12) .. controls (326.79,235.12) and (324.51,232.84) .. (324.51,230.03) -- cycle ;
\draw  [color={rgb, 255:red, 0; green, 252; blue, 245 }  ,draw opacity=1 ][fill={rgb, 255:red, 0; green, 252; blue, 245 }  ,fill opacity=1 ] (324.51,260.03) .. controls (324.51,257.22) and (326.79,254.94) .. (329.6,254.94) .. controls (332.41,254.94) and (334.69,257.22) .. (334.69,260.03) .. controls (334.69,262.84) and (332.41,265.12) .. (329.6,265.12) .. controls (326.79,265.12) and (324.51,262.84) .. (324.51,260.03) -- cycle ;
\draw [color={rgb, 255:red, 0; green, 252; blue, 245 }  ,draw opacity=1 ][fill={rgb, 255:red, 0; green, 252; blue, 245 }  ,fill opacity=1 ]   (329.6,230.03) -- (329.6,260.03) ;
\draw  [color={rgb, 255:red, 95; green, 154; blue, 131 }  ,draw opacity=1 ][fill={rgb, 255:red, 95; green, 154; blue, 131 }  ,fill opacity=1 ] (454.65,96) .. controls (454.65,93.19) and (456.93,90.91) .. (459.74,90.91) .. controls (462.55,90.91) and (464.83,93.19) .. (464.83,96) .. controls (464.83,98.81) and (462.55,101.09) .. (459.74,101.09) .. controls (456.93,101.09) and (454.65,98.81) .. (454.65,96) -- cycle ;
\draw  [color={rgb, 255:red, 95; green, 154; blue, 131 }  ,draw opacity=1 ][fill={rgb, 255:red, 95; green, 154; blue, 131 }  ,fill opacity=1 ] (454.4,125.25) .. controls (454.4,122.44) and (456.68,120.16) .. (459.49,120.16) .. controls (462.3,120.16) and (464.58,122.44) .. (464.58,125.25) .. controls (464.58,128.06) and (462.3,130.34) .. (459.49,130.34) .. controls (456.68,130.34) and (454.4,128.06) .. (454.4,125.25) -- cycle ;
\draw  [color={rgb, 255:red, 95; green, 154; blue, 131 }  ,draw opacity=1 ][fill={rgb, 255:red, 95; green, 154; blue, 131 }  ,fill opacity=1 ] (454.4,156.03) .. controls (454.4,153.22) and (456.68,150.94) .. (459.49,150.94) .. controls (462.3,150.94) and (464.58,153.22) .. (464.58,156.03) .. controls (464.58,158.84) and (462.3,161.12) .. (459.49,161.12) .. controls (456.68,161.12) and (454.4,158.84) .. (454.4,156.03) -- cycle ;
\draw  [color={rgb, 255:red, 95; green, 154; blue, 131 }  ,draw opacity=1 ][fill={rgb, 255:red, 95; green, 154; blue, 131 }  ,fill opacity=1 ] (454.4,186.03) .. controls (454.4,183.22) and (456.68,180.94) .. (459.49,180.94) .. controls (462.3,180.94) and (464.58,183.22) .. (464.58,186.03) .. controls (464.58,188.84) and (462.3,191.12) .. (459.49,191.12) .. controls (456.68,191.12) and (454.4,188.84) .. (454.4,186.03) -- cycle ;
\draw [color={rgb, 255:red, 95; green, 154; blue, 131 }  ,draw opacity=1 ][fill={rgb, 255:red, 95; green, 154; blue, 131 }  ,fill opacity=1 ]   (459.74,96) -- (459.49,125.25) ;
\draw [color={rgb, 255:red, 95; green, 154; blue, 131 }  ,draw opacity=1 ][fill={rgb, 255:red, 95; green, 154; blue, 131 }  ,fill opacity=1 ]   (459.49,156.03) -- (459.49,186.03) ;
\draw  [fill={rgb, 255:red, 0; green, 0; blue, 0 }  ,fill opacity=1 ] (194.65,96.29) .. controls (194.65,93.47) and (196.93,91.2) .. (199.74,91.2) .. controls (202.55,91.2) and (204.83,93.47) .. (204.83,96.29) .. controls (204.83,99.1) and (202.55,101.38) .. (199.74,101.38) .. controls (196.93,101.38) and (194.65,99.1) .. (194.65,96.29) -- cycle ;
\draw  [fill={rgb, 255:red, 0; green, 0; blue, 0 }  ,fill opacity=1 ] (194.4,125.54) .. controls (194.4,122.72) and (196.68,120.45) .. (199.49,120.45) .. controls (202.3,120.45) and (204.58,122.72) .. (204.58,125.54) .. controls (204.58,128.35) and (202.3,130.63) .. (199.49,130.63) .. controls (196.68,130.63) and (194.4,128.35) .. (194.4,125.54) -- cycle ;
\draw  [fill={rgb, 255:red, 0; green, 0; blue, 0 }  ,fill opacity=1 ] (194.4,156.31) .. controls (194.4,153.5) and (196.68,151.23) .. (199.49,151.23) .. controls (202.3,151.23) and (204.58,153.5) .. (204.58,156.31) .. controls (204.58,159.13) and (202.3,161.4) .. (199.49,161.4) .. controls (196.68,161.4) and (194.4,159.13) .. (194.4,156.31) -- cycle ;
\draw  [fill={rgb, 255:red, 0; green, 0; blue, 0 }  ,fill opacity=1 ] (194.4,186.31) .. controls (194.4,183.5) and (196.68,181.23) .. (199.49,181.23) .. controls (202.3,181.23) and (204.58,183.5) .. (204.58,186.31) .. controls (204.58,189.13) and (202.3,191.4) .. (199.49,191.4) .. controls (196.68,191.4) and (194.4,189.13) .. (194.4,186.31) -- cycle ;
\draw    (220.17,150.94) -- (308.35,208.76) ;
\draw [shift={(310.02,209.86)}, rotate = 213.25] [color={rgb, 255:red, 0; green, 0; blue, 0 }  ][line width=0.75]    (10.93,-3.29) .. controls (6.95,-1.4) and (3.31,-0.3) .. (0,0) .. controls (3.31,0.3) and (6.95,1.4) .. (10.93,3.29)   ;
\draw    (220.17,130.34) -- (308.37,71.02) ;
\draw [shift={(310.02,69.91)}, rotate = 146.08] [color={rgb, 255:red, 0; green, 0; blue, 0 }  ][line width=0.75]    (10.93,-3.29) .. controls (6.95,-1.4) and (3.31,-0.3) .. (0,0) .. controls (3.31,0.3) and (6.95,1.4) .. (10.93,3.29)   ;
\draw    (350.83,69.91) -- (439.02,127.73) ;
\draw [shift={(440.69,128.82)}, rotate = 213.25] [color={rgb, 255:red, 0; green, 0; blue, 0 }  ][line width=0.75]    (10.93,-3.29) .. controls (6.95,-1.4) and (3.31,-0.3) .. (0,0) .. controls (3.31,0.3) and (6.95,1.4) .. (10.93,3.29)   ;
\draw    (350.83,209.86) -- (439.03,150.54) ;
\draw [shift={(440.69,149.42)}, rotate = 146.08] [color={rgb, 255:red, 0; green, 0; blue, 0 }  ][line width=0.75]    (10.93,-3.29) .. controls (6.95,-1.4) and (3.31,-0.3) .. (0,0) .. controls (3.31,0.3) and (6.95,1.4) .. (10.93,3.29)   ;

\draw (145.33,126.86) node [anchor=north west][inner sep=0.75pt]   [align=left] {$G_1$};
\draw (480.83,126.86) node [anchor=north west][inner sep=0.75pt]   [align=left] {$G_4$};
\draw (353.33,26.63) node [anchor=north west][inner sep=0.75pt]   [align=left] {$G_2$};
\draw (353.33,230.53) node [anchor=north west][inner sep=0.75pt]   [align=left] {$G_3$};
\end{tikzpicture}
\quad\quad
\includegraphics[scale=0.28]{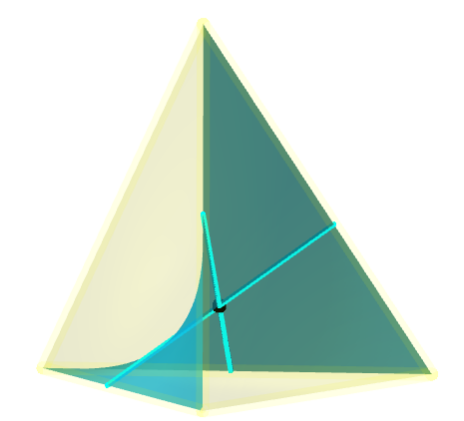}
\quad\quad
\begin{tikzpicture}[x=0.75pt,y=0.75pt,yscale=-0.65,xscale=0.65]

\draw  [color={rgb, 255:red, 95; green, 154; blue, 131 }  ,draw opacity=1 ][fill={rgb, 255:red, 95; green, 154; blue, 131 }  ,fill opacity=1 ] (261,30.2) -- (441.2,30.2) -- (441.2,210.4) -- (261,210.4) -- cycle ;
\draw [color={rgb, 255:red, 0; green, 252; blue, 245 }  ,draw opacity=1 ][line width=1.5]    (441.3,120.3) -- (260.9,120.3) ;
\draw [color={rgb, 255:red, 0; green, 252; blue, 245 }  ,draw opacity=1 ][line width=1.5]    (351.1,30.2) -- (351.1,210.4) ;
\draw  [fill={rgb, 255:red, 0; green, 0; blue, 0 }  ,fill opacity=1 ] (349.1,120.09) .. controls (349.1,118.98) and (350,118.09) .. (351.1,118.09) .. controls (352.2,118.09) and (353.1,118.98) .. (353.1,120.09) .. controls (353.1,121.19) and (352.2,122.09) .. (351.1,122.09) .. controls (350,122.09) and (349.1,121.19) .. (349.1,120.09) -- cycle ;

\draw (245.2,21.7) node [anchor=north west][inner sep=0.75pt]   [align=left] {{\footnotesize $\frac{8}{3}$}};
\draw (244.4,111.8) node [anchor=north west][inner sep=0.75pt]   [align=left] {{\footnotesize  $\frac{4}{3}$}};
\draw (245.2,212.2) node [anchor=north west][inner sep=0.75pt]   [align=left] {{\footnotesize 0}};
\draw (345.6,212.2) node [anchor=north west][inner sep=0.75pt]   [align=left] {{\footnotesize  $\frac{4}{3}$}};
\draw (435.8,212.2) node [anchor=north west][inner sep=0.75pt]   [align=left] {{\footnotesize  $\frac{8}{3}$}};
\end{tikzpicture}

    \caption{Poset of subgraphs of the $4$ vertex graph $G_4$, their CI equilibria and payoff regions.}
    \label{fig:filtration graph}
    
\end{figure}

\para 

Now, we compute the payoff region associated to the Spohn CI surface and the two Nash CI curves. In the coordinates $z_0,z_1,z_2,z_3$, the sum of all the coordinates $p_{i_1i_2i_3i_4}$ equals $\frac{1}{9}(z_0+z_1+z_2+z_3)$. We denote the cone of $\mathcal{L}$ is $\mathbb{A}^4$ by $\tilde{\mathcal{L}}$. Then, we identify the set of totally mixed  CI equilibria $\mathcal{L}\cap\Delta_3^\circ$ with the intersection of $\tilde{\mathcal{L}}$ and the open simplex 
$$\tilde{\Delta}^\circ=\{(z_0,z_1,z_2,z_3)\in \mathbb{A}^4:z_0+z_1+z_2+z_3=\frac{1}{9} \text{ and }z_0,z_1,z_2,z_3> 0\}.$$

In the coordinates $z_0,z_1,z_2,z_3$, the restriction of the payoff map to the set of totally mixed CI equilibria is 
\[
\begin{array}{cccl}
    \pi_X:&\tilde{\Delta} \cap\tilde{\mathcal{L}}&\longrightarrow  &\mathbb{R}^4\\
     & (z_0,z_1,z_2,z_3)& \longmapsto&\left(PX^{(1)} ,PX^{(2)} ,PX^{(3)} ,PX^{(4)}
     \right)
\end{array},
\]
where the expected payoffs are 
$$
\begin{array}{cc}
PX^{(1)}= 24(z_0+z_2),&PX^{(2)} = -24(z_1+z_3) \\
PX^{(3)} = 24(z_0+z_1) , &PX^{(4)} = -24(z_2+z_3)
\end{array}.$$
Note that $PX^{(1)}+PX^{(2)}=PX^{(3)}-PX^{(4)}= \frac{8}{3}$. Therefore, we can consider the payoff map $\pi_X$ as the map from $\tilde{\Delta} \cap\tilde{\mathcal{L}}$ to $\mathbb{R}^2$ sending $(z_0,z_1,z_2,z_3)$ to $(PX^{(1)},PX^{(3)})$.
Restricting the Segre parametrization to $\Delta$, we obtain the following parametrization of $\tilde{\Delta}^\circ \cap\tilde{\mathcal{L}}$:
\[
\begin{array}{cccc}
\varphi:&\mathbb{R}_{>0}^2&\longrightarrow& \tilde{\Delta} \cap\tilde{\mathcal{L}}\\
& (a,b)&\longmapsto &\left( \frac{ab}{9(a+1)(b+1)},\frac{a}{9(a+1)(b+1)},\frac{b}{9(a+1)(b+1)},\frac{1}{9(a+1)(b+1)}, \right)
\end{array}.
\]
In particular, we get that the composition $\pi_X\circ\varphi$ sends a point $(a,b)\in\mathbb{R}^2_{>0}$ to $\frac{8}{3}(\frac{b}{b+1},\frac{a}{a+1})$. Therefore, the CI payoff region equals the open square $(0,\frac{8}{3})\times(0,\frac{8}{3})$ in $\mathbb{R}^2$.
Similarly, for the two Nash CI curves, we get that the payoff regions are the open intervals $\{\frac{4}{3}\}\times(0,\frac{4}{3})$ and $(0,\frac{4}{3})\times \{\frac{4}{3}\}$ respectively as illustrated in Figure \ref{fig:filtration graph}. In these coordinates, the totally mixed Nash point is $(\frac{1}{36},\frac{1}{36},\frac{1}{36},\frac{1}{36})$. The corresponding expected payoff in $\mathbb{R}^2$ is the point $(\frac{4}{3},\frac{4}{3})$.
In particular, we see that in the image of Nash CI curves there are totally mixed CI equilibria that give better expected payoffs than the totally mixed Nash equilibria. For instance, the points $(\frac{3}{72},\frac{3}{72},\frac{1}{72},\frac{1}{72})$
and $(\frac{3}{72},\frac{1}{72},\frac{3}{72},\frac{1}{72})$
lie in the two Nash CI curves respectively and they give better expected payoff than the totally mixed Nash equilibria. Similarly, for any totally mixed Nash CI equilibria on a Nash CI curve, there exists a totally mixed CI equilibria in $\mathcal{L}$ which gives better expected payoffs.
    
\end{example}

\section{Nash conditional independence varieties}\label{sec: Nash CI}

The goal of this section is to analyze the algebro-geometric properties of the Spohn CI variety of undirected graphical models whose connected components are all cliques. Let $(s_1,\ldots,s_k)$ be a partition of the set $[n]$ with $\emptyset \neq s_i\subseteq [n]$ and $|s_i| = n_i$. Given such a partition, we consider the complete graphs $G_1,\ldots,G_k$ on the set of vertices $s_1,\ldots,s_k$ respectively. Note that up to the labeling of the vertices, the integers $n_1,\ldots,n_k$ carry all the information of the partition. Thus, we denote the partition as $\mathbf{n}:= (n_1,\ldots,n_k)$ where $1 \leq n_1 \leq \cdots \leq n_k \leq n$ throughout the section. This modeling can be seen as players forming $k$ groups, where each group's members act dependently within the group but independently from all other players. In Example~\ref{ex:nash variety}, we studied indeed such models for 4-player games. We define the graph $G_\mathbf{n} : = G_1 \sqcup \cdots \sqcup G_k$ and denote the discrete conditional independence model of $G_\mathbf{n}$ by $\Mn$. We first compute the independence model $\Mn$.
In the following proposition, we see how connected components of any undirected graphical model
$G=([n],E)$ get translated to products in the (not necessarily the positive part) discrete conditional independence model.

\begin{proposition}\label{prop:prod graphs}
Let $G =([n],E)$ be an undirected graphical model with $k$ connected components $G_i$. Then
\[
\Mglobal = \mathcal{M}_{\text{global}(G_1)}\times\cdots\times \mathcal{M}_{\text{global}(G_k)}.
\]
\end{proposition}
\begin{proof}
For simplicity, we will assume that $G$ has two connected components, $G_1=([n_1],E_1)$ and $G_2 = ([n_2],E_2)$. Then, we have that $[n_1]\independent [n_2]|\emptyset \in \text{global}(G)$. By (\ref{eq: independence ideal}), we obtain the corresponding independence ideal equals the ideal defining the Segre variety $\P^{d_{i_1} \cdots \ d_{i_{n_1}}-1}\times\P^{d_{j_1} \cdots \ d_{j_{n_2}} -1}$. In particular, $\Mglobal\subseteq \P^{d_{i_1} \cdots \ d_{i_{n_1}}-1}\times\P^{d_{j_1} \cdots \ d_{j_{n_2}} -1}$. Consider the following parametrization of the Segre variety
\begin{equation}\label{eq: segre par in the proof}
p_{i_1\cdots i_{n_1}j_1\cdots j_{n_2}} = 
\sigma^{(1)}_{i_1\cdots i_{n_1}} \sigma^{(2)}_{j_1\cdots j_{n_2}}.
\end{equation}

If $\mathcal{X}_A\independent \mathcal{X}_B \ |\ \mathcal{X}_C \in \text{global}(G_1)$, then $\mathcal{X}_A\independent \mathcal{X}_B \ |\ \mathcal{X}_C \sqcup [n_2] \in \text{global}(G)$. In particular, evaluating the Segre parametrization (\ref{eq: segre par in the proof}) in the independence ideal of the latter CI statement we deduce that $\mathcal{M}_{\text{global}(G)} \subseteq \mathcal{M}_{\text{global}(G_1)}\times \P^{d_{j_1} \cdots \ d_{j_{n_2}} -1}$. Similarly, we get that $\mathcal{M}_{\text{global}(G)}\subseteq \P^{d_{i_1} \cdots \ d_{i_{n_1}}-1}\times \mathcal{M}_{\text{global}(G_2)}$, and hence, $\mathcal{M}_{\text{global}(G)} \subseteq \mathcal{M}_{\text{global}(G_1)}\times\mathcal{M}_{\text{global}(G_2)}$. On the other hand, every CI statement in $\text{global}(G)$ is of the form $\mathcal{X}_{A_1\sqcup B_1}\independent \mathcal{X}_{A_2 \sqcup B_2}|\mathcal{X}_{A_3 \sqcup B_3}$, where $\mathcal{X}_{A_1}\independent \mathcal{X}_{A_2}| \mathcal{X}_{A_3}$ and $\mathcal{X}_{B_1}\independent \mathcal{X}_{B_2}| \mathcal{X}_{B_3}$ are in $\text{global}(G_1)$ and $\text{global}(G_2)$ respectively. By axioms $C1$ and $C2$ and definition of separation on undirected graphical models \cite[page 29]{Lau}, it is enough to consider CI statements where $\bigsqcup_{i \in [3]} A_i \sqcup \bigsqcup_{j \in [3]} B_i = [n]$. By (\ref{eq: independence ideal}), the quadrics generating the corresponding independence ideal of $\mathcal{X}_{A_1\sqcup B_1}\independent \mathcal{X}_{A_2 \sqcup B_2}|\mathcal{X}_{A_3 \sqcup B_3}$ are of the form
\begin{equation}\label{eq:quadric product}
   p_{abc\,\alpha\beta\gamma}\cdot p_{a'b'c\,\alpha'\beta'\gamma}-p_{a'bc\,\alpha'\beta\gamma}\cdot p_{ab'c\,\alpha\beta'\gamma} 
\end{equation}

for $a,a'\in \mathcal{R}_{A_1}$, $b,b'\in \mathcal{R}_{A_2}$, $c\in \mathcal{R}_{A_3}$, $\alpha,\alpha'\in \mathcal{R}_{B_1}$, $\beta,\beta'\in \mathcal{R}_{B_2}$, and $\gamma\in \mathcal{R}_{B_3}$. We claim that such quadric lies inside the ideal of $\mathcal{M}_{\mathcal{C}_1}\times\mathcal{M}_{\mathcal{C}_2}$. Indeed. the quadrics 
\[\begin{array}{c}
q_1 = \sigma^{(1)}_{abc}\sigma^{(1)}_{a'b'c}-\sigma^{(1)}_{a'bc}\sigma^{(1)}_{ab'c},\\
q_2 = \sigma^{(2)}_{\alpha\beta\gamma}\sigma^{(2)}_{\alpha'\beta'\gamma}-\sigma^{(2)}_{\alpha'\beta\gamma}\sigma^{(2)}_{\alpha\beta'\gamma}.
\end{array}
\]
lie in the ideal of $\mathcal{M}_{\mathcal{C}_1}\times\mathcal{M}_{\mathcal{C}_2}$.
Then, the expression 
\[
\sigma^{(2)}_{\alpha\beta\gamma}\sigma^{(2)}_{\alpha'\beta'\gamma}q_1   +\sigma^{(1)}_{a'bc}\sigma^{(1)}_{ab'c}q_2
\]
coincides with the evaluation of the Segre parametrization \eqref{eq: segre par in the proof} in the quadric \eqref{eq:quadric product}. Hence, we conclude that $\Mc =  \mathcal{M}_{\mathcal{C}_1}\times\mathcal{M}_{\mathcal{C}_2}$.
\end{proof}

\para Proposition \ref{prop:prod graphs} allows us to compute the discrete conditional independence model of the graphical model $G_\mathbf{n}$ with binary choices for a partition $\mathbf{n}$ of the set $[n]$, which is the focus of this section. Alternatively, since $G_\mathbf{n}$ is decomposable, one can conclude the following corollary by Proposition~\ref{prop: parametrized model}
and \cite[Theorem 4.2]{GMS06}.

\begin{corollary}\label{co:segre graph}
Let $\mathbf{n}=(n_1,\ldots,n_k)$ be a partition of the set $[n]$, and let  $G_\mathbf{n}$ be a binary undirected graphical model whose connected components are $G_1,\ldots,G_k$, where each $G_i$ is a complete graph on $n_i$ vertices. Then, $\Mn = \P^{2^{n_{1}}-1}\times\cdots\times \P^{2^{n_k}-1}$ is the Segre variety.
\end{corollary}

\begin{example}\label{ex:indep 2 2}
    Consider the partition $(\{1,2\},\{3,4\})$ of the set $[4]$, i.e.\ $\mathbf{n} = (2,2)$. The corresponding graph $G_4$ is the graph appearing in Figure \ref{fig:filtration graph} which is the disjoint union of two cliques on 2 vertices. The corresponding independence model is $\mathcal{M}_{\mathbf{n}}= \P^{3}\times\P^3\subset \P^{15}$.
\end{example}

We define the \textit{$\mathbf{n}$--Nash conditional independence (CI) variety}, denoted by $N_{X,\mathbf{n}}$, as the Spohn CI variety of $G_{\mathbf{n}}$. We define the set of \textit{totally mixed $\mathbf{n}$--Nash conditional independence (CI) equilibria} to be  the intersection of $\Nxn\cap \left(\Delta_1\times\cdots\times\Delta_k\right)$, where $\Delta_i$ is the open simplex of the corresponding factor of the Segre variety $\Mn$. In other words, the set of totally mixed $\mathbf{n}$--Nash CI equilibria is the set of totally mixed CI equilibria for the independence model $\Mn$.

\para 

\begin{example}
Certain totally mixed $\n$-Nash CI equilibria are already well-known.
\begin{itemize}
    \item For $\n= (1,\ldots,1 )$, we get that the intersection of $\Nxn$ with the open simplex is the set of totally mixed Nash equilibria. Hence, the set of totally mixed Nash equilibria and the set of totally mixed Nash CI equilibria coincide.
    \item For $\n=(1,\ldots,1,2)$, $\Nxn$ is the Nash CI curve.
    \item For $\n = (n)$, $\Nxn$ is the Spohn variety and  the set of totally mixed Nash CI equilibria is the set of totally mixed dependency equilibria.
\end{itemize}
\end{example}

\para 

As a consequence of Theorem \ref{theo:conj decom} and Corollary~\ref{co:segre graph}, we deduce the following result. 

\para 

\begin{proposition}\label{pro:dim nash}
Let $\n=(n_1,\ldots,n_k)$ a partition of $[n]$. Then, for generic payoff tables, the dimension of $\Nxn$ is 
\[
\dim \Nxn = \dim \Mn - n =  2^{n_1}+\cdots + 2^{n_k}-k-n.
\]
\end{proposition}

\para 

Note that Proposition \ref{pro:dim nash} agrees with \cite{IJ}[Proposition 4] for the case of Nash CI curves. Moreover, this is the only case where the $\mathbf{n}$--Nash CI variety is a curve for generic payoff tables. Similarly, the only possible case where 
the $\mathbf{n}$--Nash CI variety is a surface for generic payoff tables whenever $\mathbf{n}=(1,\ldots,1,2,2)$ or $\mathbf{n} = (2,2)$.  
  We say that the Nash CI variety $\Nxn$ is a \textit{Nash CI surface} if it is of dimension $2$ and $\mathbf{n} = (1,\ldots,1,2,2)$ or $\mathbf{n} = (2,2)$. 
\begin{example}
For the partition $(\{1,2\},\{3,4\})$ in Example \ref{ex:indep 2 2}, we have that for generic payoff tables, the Nash CI variety is a Nash CI surface. In Example \ref{ex:nash variety} we illustrated a concrete example of a Nash CI surface associated to the partition $(\{1,2\},\{3,4\})$ and we computed the $(2,2)$--Nash CI equilibria.
  On the other hand, in Example \ref{ex:filtration} for generic $3$--players games, the Spohn CI variety of a graph with three vertices and two edges is a surface. However, since the graph is connected but not complete, it is not a Nash CI surface.     
\end{example}

\subsection{Equations of Nash CI varieties}
In Proposition \ref{pro:dim nash} we saw that a generic Nash CI variety has codimension $n$ in a Segre variety. Now, we improve this result by showing that a generic Nash CI variety is a complete intersection in $\Mn$. This will allow us to compute some properties and invariants Nash CI varieties such as their degree. 
To do so, first, we present the equations defining $\Nxn$ inside $\Mn$. We follow the same strategy as in Section \ref{sec: dimension of Spohn CI}. We evaluate the equations of the Spohn variety at the parametrization of the Segre variety and we remove the factors that lead to components in the hyperplanes we are saturating by. There, we computed the polynomials $F_1,\ldots,F_n$ and we considered the variety $Y_X$ defined by them. The restriction of these polynomials to our particular case provides the equations of $N_X$.\\

Given a partition  $\n=(n_1,\ldots,n_k)$ of $[n]$, we label the $n$ players of the game by $(1,1),\ldots,(1,n_1),\ldots,(k,1),\ldots,(k,n_k)$. We denote the payoff tables of the game by $X^{(1,1)},\cdots, X^{(1,n_1)}, \cdots, X^{(k,1)}, \cdots ,X^{(k,n_k)}$.
We consider the parametrization of $\Mn$ by Proposition~\ref{prop: parametrized model} given by 
\begin{equation}\label{eq: parametrization for Nash CI}
p_{j_{11}\cdots j_{1 n_1} \cdots j_{k1} \cdots j_{kn_k}} := \sigma^{(1)}_{j_{11}\cdots j_{1n_1}}\cdots \sigma_{j_{k1}\cdots j_{kn_k}}^{(k)}
\end{equation}
where $j_{i l_i}\in [2] $ for $i \in [k]$ and $l_i\in[n_i]$.
Evaluating the $2 \times  2$ minors i.e.\ the determinant of $M_{(i,l_i)}$ at this parametrization we obtain

\begin{equation}\label{eq:poly gen det}
F_{(i,l_i)}:=\mathrm{det}
\left(
\begin{array}{cc}
\!\!\!
\displaystyle \sum_{j_{i1}\cdots \widehat{j_{il_i}} \cdots j_{in_i}} \!\!\!\!\!\!\!\!\sigma_{j_{i1}\cdots 1 \cdots j_{in_i}}^{(i)} & \!\!\!\!
\displaystyle\sum_{j_{11}\cdots \widehat{j_{il_i}}\cdots j_{kn_k}} \!\!\!\!\!\!\!\!X^{(i,l_i)}_{j_{11}\cdots 1\cdots j_{kn_k}} \sigma_{j_{11}\cdots j_{1n_1}}^{(1)}\!\!\!\!\cdots\sigma_{j_{i1}\cdots 1\cdots j_{in_i}}^{(i)}\!\!\!\!\cdots \sigma_{j_{k1}\cdots  j_{kn_k}}^{(k)}\!\!\!

\\ 
\!\!\!
\displaystyle \sum_{j_{i1}\cdots \widehat{j_{il_i}} \cdots j_{in_i}} \!\!\!\!\!\!\!\!\sigma_{j_{i1}\cdots 2 \cdots j_{in_i}}^{(i)}&\!\!\!\!
\displaystyle\sum_{j_{11}\cdots \widehat{j_{il_i}}\cdots j_{kn_k}} \!\!\!\!\!\!\!\!X^{(i,l_i)}_{j_{11}\cdots 2\cdots j_{kn_k}} \sigma_{j_{11}\cdots j_{1n_1}}^{(1)}\!\!\!\!\cdots\sigma_{j_{i1}\cdots 2\cdots j_{in_i}}^{(i)}\!\!\!\!\cdots \sigma_{j_{k1}\cdots  j_{kn_k}}^{(k)}\!\!\!

\end{array}
\right)
\end{equation}
with the product of \[\left(\displaystyle \sum_{j_{i1}\cdots \widehat{j_{i l_i }} \cdots j_{i n_i}} \sigma^{(i)}_{j_{i1}}\cdots \widehat{\sigma^{(i)}_{j_{i l_i}}} \cdots \sigma^{(i)}_{j_{i n_i}}\right).
\]

We are interested in the polynomial $F_{(i,l_i)}$, since we saturate the resulting ideal by the hyperplanes $ \{
p_{+\cdots + j_{i l_i} + \cdots +}=0\}$. Note that for $n_i = 1$ i.e.\ the associated clique consists of one vertex, thus we obtain the familiar equation from studying totally Nash equilibria (\cite[Chapter 6]{CBMS}):

\begin{equation}\label{eq:poly n=1}
F_{(i,1)} = 
\displaystyle\sum_{j_{11}\cdots \widehat{j_{i 1}} \cdots j_{kn_k}} \left(X^{(i,1)}_{j_{11}\cdots 2 \cdots j_{kn_k}} -X^{(i,1)}_{j_{11}\cdots 1 \cdots j_{kn_k}}\right)\sigma_{j_{11}\cdots  j_{1n_1}}^{(1)}\cdots\widehat{\sigma_{j_{i 1}}^{(i)}}\cdots \sigma_{j_{k1}\cdots  j_{kn_k}}^{(k)}.
\end{equation}
Then, for generic payoff tables $X^{(1,1)},\ldots,X^{(k,n_k)}$, we deduce that
\begin{equation}\label{eq:nash CI equ}
N_{X,\mathbf{n}} \subseteq \mathbb{V}(F_{(1,1)},\ldots, F_{(k,n_k)} )\subseteq \Mn.
\end{equation}
Note that for $\mathbf{n} = (1,\ldots,1,2)$, the polynomials $F_{(1,1)},\ldots,F_{(kn_k)}$ coincide with the polynomials defining the Nash CI curve computed in \cite{IJ}[Section 2.2]. For instance, the equations shown in Example \ref{ex:nash variety} are obtained from Equation \ref{eq:poly gen det}. 

\para 

Next, we show that for generic payoff tables, the inclusion \eqref{eq:nash CI equ} is an equality. 
Let $D_{(i,j_i)}$ be the divisor in $\Mn$ defined by the polynomial $\mathbb{V}(F_{(i,j_i)})$. The divisor $D_{(i,j_i)}$ lies in the linear system defined by the line bundle 
\begin{equation}\label{eq:line bundle Nash}
\O(1,\ldots,1,\underset{(i)}{(1-\delta_{1,n_i})2},1,\ldots,1))
\end{equation}
of $\Mn$,
where $\delta_{i,j}$ is $1$ if $i=j$, and $0$ if $i\neq j$. 
In other words,
\[
\O(1,\ldots,1,\underset{(i)}{(1-\delta_{1,n_i})2},1,\ldots,1))=\left\{
\begin{array}{cl}
\O(1,\ldots,1,\underset{(i)}{0},1,\ldots,1))   & \text{ if } n_i = 1,\\
\noalign{\vspace*{4mm}}
 \O(1,\ldots,1,\underset{(i)}{2},1,\ldots,1))  & \text{ if } n_i>1.
\end{array}
\right.
\]

Now, we consider the map that sends a payoff table $X^{(i,j_i)}$ to the divisor $D_{(i,j_i)}$. More precisely, for $(i,j_i)$, we consider the map 
\begin{equation}\label{eq:div map}
\begin{array}{cccl}
\phi_{(i,j_i)}:& \mathbb{R}^{2^n}&\longrightarrow& H^0(\Mn,\O(1,\ldots,1,\underset{(i)}{(1-\delta_{1,n_1})2},1,\ldots,1))\\
 & X^{(i,j_i)}&\longmapsto& F_{(i,j_i)}
\end{array}.
\end{equation}

We denote the image of $\phi_{i,j_i}$ by $\Lambda_{(i,j_i)}$. 
In Section \ref{sec: dimension of Spohn CI} we studied these linear systems. In particular, in Lemma \ref{lemma:gens lin syst} the generators of $\Lambda_{(i,j_i)}$ were computed. The next result is the translation of Lemma \ref{lemma:gens lin syst} to the setting of Nash CI varieties.

\para 

\begin{lemma}\label{lemma:gen system Nash CI}
For $n_i=1$, the linear system $\Lambda_{(i,1)}$ is complete. For $n_i\geq 2$, we have that 
\[
\Lambda_{(i,l_i)}\simeq W_{(i,l_i)}\otimes \bigotimes_{j\neq i} H^0(\P^{2^{n_j}-1},\O(1)),
\]
where $ W_{(i,l_i)}$ is the linear system of $\P^{2^{n_i}-1}$ generated by the polynomials
\begin{enumerate}
    \item 
    for $(j_1,\ldots,\widehat{j_{l_i}},\ldots,j_{n_i})\in [2]^{n_i-1}$,
    $\displaystyle\sigma^{(i)}_{j_{1}\cdots 1\cdots j_{n_i}} \left(\sum_{m_1,\ldots,\widehat{m_{l_i}},\ldots,m_{n_i}}
    \sigma^{(i)}_{m_{1}\cdots 2\cdots m_{n_i}}
    \right)$, 
    \item for $(m_1,\ldots,\widehat{m_{l_i}},\ldots,m_{n_i})\in [2]^{n_i-1}$,
    $\displaystyle\sigma^{(i)}_{m_{1}\cdots 2\cdots m_{n_i}} \left(\sum_{j_1,\ldots,\widehat{j_{l_i}},\ldots,j_{n_i}}
    \sigma^{(i)}_{j_{1}\cdots 1\cdots j_{n_i}}
    \right)$, 
    \item $ 
    \displaystyle \sigma^{(i)}_{1\cdots 1\cdots 1}\sigma^{(i)}_{ 1\cdots 2\cdots 1}
    -
\sum_{{\scriptsize 
		\begin{array}{c}
		(j_{1}\cdots \widehat{j_{l_i}}\cdots j_{n_i})\neq (1,\dots,1)\\
		(m_{1}\cdots \widehat{m_{l_i}}\cdots m_{n_i})\neq (1,\dots,1)
		\end{array}}}\!\!\!\!\!\!\!\!\!\!\!\!\!\!
\sigma^{(i)}_{j_{1}\cdots 1\cdots j_{n_i}}
\sigma^{(i)}_{m_{1}\cdots 2\cdots m_{n_i}}.
    $
\end{enumerate}
    
\end{lemma}

\para 

We deduce that for $n_i>1$, the map $\phi_{i,j_i}$ is not surjective. 

\para 

\begin{corollary}\label{co:system dim}
For $(i,l_i)$ such that $n_i>1$, $\Lambda_{(i,l_i)}$ has dimension $ (2^{n_i}-1)\displaystyle\prod_{j\neq i} 2^{n_j}$.
\end{corollary}

Using Lemma \ref{lemma:gen system Nash CI}, we derive the following result.

\para 

\begin{proposition}\label{prop:eq Nash CI}
    For generic payoff tables, $$\Nxn=\bbV(F_{(1,1)},\cdots,F_{(k,n_k)}).$$ In particular, for generic payoff tables, $\Nxn$ is a complete intersection in $\Mn$.
\end{proposition}
\begin{proof}
    We consider the variety 
    \[
    \mathcal{X}:=\{(X,p)\in V^{n}\times \P^{2^{n}-1}:p\in\bbV(F_{(1,1)},\cdots,F_{(k,n_k)})\}
    \]
    together with the projection $\pi:\mathcal{X}\rightarrow V^{n}$. Here, we identify $X\in V^n$ with the game $X=(X^{(1,1)},\ldots,X^{(k,n_k)})$. We denote the fiber of $X$ via $\pi$ by $\mathcal{X}_X$. Note that $\pi$ is surjective and, for any $X$, $\dim \mathcal{X}_X\geq \dim \Mn -n$.
    Let $H$ be a hyperplane of $\P^{2^n-1}$ of the form $\{p_{j_1j_2 \cdots j_n}=0\}$ or $\{p_{++ \cdots +} = 0\}$. We consider the variety
    $$\Sigma_{H}:=\overline{\mathcal{X}\setminus (V^n\times H)}.$$
    For $X\in V^n$, we denote the intersection of $\mathcal{X}_X$ with $\Sigma_{H}$ by $\Sigma_{H,X}$. Note that 
    for $X\in V^n$, $\Sigma_{H,X}$ contains the Nash CI variety $\Nxn$. By Theorem \ref{theo:conj decom},
    the restriction of $\pi$ to $\Sigma_{H}$ is dominant. 
    We want to show that $\mathcal{X}_X$ equals $\Nxn$ for generic $X\in V^n$. This is equivalent to show that for any hyperplane $H$ of the form $\{p_{j_1j_2 \cdots j_n}=0\}$ or $\{p_{++ \cdots +} = 0\}$, we have that  $\Sigma_{H,X} =\mathcal{X}_X$ for generic $X\in V^n$.

    \para 

     If $\mathcal{X}$ has no irreducible component in $V^n\times H$, then $\Sigma_{H}$ is dense in $\mathcal{X}$. Hence, $\Sigma_{H,X} =\mathcal{X}_X$ for generic $X$.
    Assume now that $\mathcal{X}$ has an irreducible component contained in $V^n\times H$. Let $\mathcal{X}_1$ be the union of these irreducible components.
    If the restriction of $\pi$ to $\mathcal{X}_1$ is not dominant, then $\Sigma_{H,X} =\mathcal{X}_X$ for generic $X$. Assume that $\pi|_{\mathcal{X}_1}$ is dominant. 
    Since $\pi|_{\mathcal{X}_1}$ is closed, it is surjective.
    Assume that there exists $X\in V^n$ such that $\mathcal{X}_X$ has dimension $ \dim \Mn -n$ and $\mathcal{X}_X$ has no irreducible component contained in $H$. Then, the intersection of $\mathcal{X}_X$ and $\mathcal{X}_1$ has dimension at most $\dim \Mn -n-1$. Using that the dimension of the fibers of $\pi|_{\mathcal{X}_1}$ is upper semicontinuous, we get that the generic fiber of $\pi|_{\mathcal{X}_1}$ has dimension at most $\dim \Mn -n-1$. This is a contradiction since the dimension of the fibers of $\pi|_{\mathcal{X}_1}$ is at least $\dim \Mn -n$. 

\para 

    Therefore, it is enough to show that  there exists $X\in V^n$ such that $\mathcal{X}_X$ has dimension $ \dim \Mn -n$ and $\mathcal{X}_X$ has no irreducible component contained in $H$. Assume that $H$ is defined by $p_{m_{(1,1)}\cdots m_{(k,n_k)}}=0$ for fixed $m_{(1,1)},\ldots,m_{(k,n_k)}\in[2]$. By Lemma \ref{lemma:gen system Nash CI}, we can choose $X$ such that 
    \[
    F_{(i,l_i)}=q_{(i,l_i)}\displaystyle\prod_{j\neq i}^kl^{(i,l_i)}_j,
    \]
    where $q_{(i,l_i)}$ is an element in the linear system $W_{(i,l_i)}$ and $l_j^{(i,l_i)}$ is a generic element of the complete linear system $H^0(\P^{2^{n_j}-1},\O_{\P^{2^{n_j}-1}}(1))$. Then, the irreducible components of $\mathcal{X}_X$ are defined by $n$ polynomials of the form $q_{(i,l_i)}$ or $l_j^{(i,l_i)}$. Since the linear forms $l_j^{(i,l_i)}$ are generic in a complete linear system, by Bertini's Theorem (see \cite[Theorem 8.18]{Hart}), it is enough to check that the intersection of any number of quadrics of the form $q_{(i,l_i)}$ has the expected dimension and none of its irreducible components is contained in $H$. This can be checked on each of the factors of the Segre variety $\Mn$. In other words, given a player $i$, we need to check that there exist $q_{(i,1)}\in W_{(i,1)},\ldots,q_{(i,n_i)}\in W_{(i,n_i)}$ such that for any subset $S\subseteq [n_i]$, the variety $$\bbV(q_{(i,l)}:l\in S)\subset\P^{2^{n_i}-1}$$ is a complete intersection and it has no irreducible component in the hyperplane $H_i:=\{\sigma^{(i)}_{m_{(i,1)}\cdots m_{(i,n_i)}}=0\}$. For simplicity, we will assume that $S=[n_i]$. The same arguments can be apply to any subset of $[n_i]$

\para 

    For a player $(i,l)$, we fix the index 
    \[
    \widetilde{m}_{(i,l)}=\left\{\begin{array}{cl}
1 & \text{ if } m_{(i,l)}=2\\
2 & \text{ if } m_{(i,l)}=1
\end{array}.
    \right.
    \]
    By Lemma \ref{lemma:gen system Nash CI}, we can set the quadric $q_{(i,l)}$ to be the product
    \begin{equation}\label{eq:some quad}
    \displaystyle
q_{(i,l)}=\sigma^{(i)}_{\widetilde{m}_{(i,1)}\cdots\underset{(l)}{m_{(i,l)}}\cdots \widetilde{m}_{(i,n_i)}}\left(
\sum_{a_1,\ldots,\widehat{a_l},\ldots,a_{n_i}}\sigma^{(i)}_{a_1\cdots\underset{(l)}{\widetilde{m}_{(i,l)}}\cdots a_{n_i}}
\right)
    \end{equation}
    We denote the linear forms in \eqref{eq:some quad} by $\mathfrak{S}_{(i,l)}$ and $g_{(i,l)}$ respectively. Up to labelling of the players, the irreducible components of $\bbV(q_{(i,1)},\ldots,q_{(i,n_i)})$ are linear subspaces of the form $\bbV(\mathfrak{S}_{(i,1)},\ldots,\mathfrak{S}_{(i,j)})\cap\bbV(g_{(i,j+1)},\ldots,g_{(i,n_i)})$ for $j\leq n_i$. First of all, note that $\bbV(\mathfrak{S}_{(i,1)},\ldots,\mathfrak{S}_{(i,j)})$ has the expected dimension since its the zero locus of $j$ distinct monomials. Now, for $l>j$, the monomial $\sigma^{(i)}_{m_{(i,1)}\cdots\widetilde{m}{(i,l)}\cdots m_{(i,n_i)}}$ appears in $g_{(i,l)}$ and it does not appear in any of the other linear forms $g_{(i,j+1)},\ldots,g_{(i,n_i)}$. This implies that $\bbV(g_{(i,j+1)},\ldots,g_{(i,n_i)})$ has also the expected dimension. Moreover, the monomial $\sigma^{(i)}_{m_{(i,1)}\cdots\widetilde{m}{(i,l)}\cdots m_{(i,n_i)}}$ does not appear neither in the linear forms $\mathfrak{S}_{(i,1)},\ldots,\mathfrak{S}_{(i,j)}$. 
    Thus, the intersection $$\bbV(\mathfrak{S}_{(i,1)},\ldots,\mathfrak{S}_{(i,j)})\cap\bbV(g_{(i,j+1)},\ldots,g_{(i,n_i)})$$ has the expected dimension. It remains to show that this intersection is not contained in  $H_i=\{\sigma^{(i)}_{m_{(i,1)}\cdots m_{(i,n_i)}}=0\}$. This follows from the fact that the variable $\sigma^{(i)}_{m_{(i,1)}\cdots m_{(i,n_i)}}$ does not appear in the linear forms $\mathfrak{S}_{(i,1)},\ldots,\mathfrak{S}_{(i,j)},g_{(i,j+1)},\ldots,g_{(i,n_i)}$. We conclude that for $q_{(i,l)}$ as in \eqref{eq:some quad} and for generic linear forms $l_j^{(i,l)}$, $\dim\mathcal{X}_X=\dim \Mn-n$ and $\mathcal{X}_X$ has no irreducible components contained in the hyperplane $H$. Therefore, for generic $X\in V^n$, we have that $\mathcal{X}_X =\Sigma_{H,X}$. 

    \para 
    
    A similar argument shows that the same holds for hyperplanes of the form $\{p_{++ \cdots +} = 0\}$. Since there are only a finite number of hyperplanes of this form, we deduce that for generic $X\in V^n$, $\mathcal{X}_X$ has no irreducible component included in these hyperplanes. We conclude that $\mathcal{X}_X$ equals the Spohn CI variety $\Vc$ for generic payoff tables.
\end{proof}

\subsection{Algebro-geometric properties}\label{sec:smoothness}
By Propositions \ref{prop:eq Nash CI}, $\Nxn$ is the complete intersection of the divisors $D_{(1,1)},\ldots,D_{(k,n_k)}$. Recall that $D_{(i,l)}$ is the divisor defined by $F_{(i,l)}$ and it lies in the linear system given by $\Lambda_{(i,l)}$.
Now we compute the degree of generic Nash CI varieties. 

\para 

\begin{proposition}\label{prop:degree}
For generic payoff tables, the degree of $\Nxn$ is the coefficient of the monomial 
\[
x_1^{2^{n_1}-1}\cdots x_k^{2^{n_k}-1}
\]
in the polynomial 
\begin{equation}\label{eq:poly deg}
\displaystyle
\left( 
\sum_{i=1}^kx_i
\right)^{\sum_i 2^{n_i}-n-k}
\prod_{\beta = 1}^k\left( 
\sum_{i=1}^kx_i +(-1)^{\delta_{1,n_\beta}}x_\beta
\right)^{n_\beta},
\end{equation}
where $ \delta_{i,\beta}=\left\{   \begin{array}{cl}
0 & \text{ if } i\neq \beta\\
1 & \text{ if } i=\beta
\end{array}
\right.$.
\end{proposition}
\begin{proof}
    We compute the degree of $ \Nxn$ using the Chow ring of $\Mn$, which is given by
    \begin{equation}\label{eq:chow ring}
    \mathcal{A}_{\bullet}(\Mn) \simeq \Z[x_1,\ldots,x_k]/\langle x_1^{2^{n_1}},\ldots, x_k^{2^{n_k}}  \rangle.
    \end{equation}
    The line bundle of the divisor $D_{(\beta,j_\beta)}$ is 
    \[
    \O_{\Mn}(D_{(\beta,j_\beta)}) =
    \O(1,\ldots,1,\underset{(i)}{(1-\delta_{1,n_\beta})2},1,\ldots,1))
    .
    \]
    Thus, we get that the class of $ D_{(\beta,j_\beta)}$ in $\mathcal{A}_{\bullet}(\Mn)$ is given by
    \[
    \left[ D_{(\beta,j_\beta)}\right]= \displaystyle\sum_{i=0}^kx_i +(-1)^{\delta_{1,n_\beta}}x_\beta.
    \]
    Let $H$ be a generic hyperplane. We deduce that the product $[D_{(1,1)}]\cdots [D_{(k,n_n)}][H\cap \Mn]$ is equal to the class of polynomial \eqref{eq:poly deg} in $\mathcal{A}_{\bullet}(\Mn) $. Thus, we conclude that the degree of $\Nxn$ is the coefficient of the monomial $x_1^{2^{n_1}-1}\cdots x_k^{2^{n_k}-1}$ of the polynomial \eqref{eq:poly deg}. We refer to \cite[Chapters 1 and 2]{EH} for more details on
    this computation.
\end{proof}

\para 

In the case of a Nash CI curve, some of the most important invariants are the degree and the genus. In \cite{IJ}[Section 3] these invariants were computed. In particular, Proposition \ref{prop:degree} for Nash CI curves coincides with \cite{IJ}[Lemma 7]. For algebraic surfaces, one important invariant is the Kodaira dimension, which plays a fundamental role in the classification of smooth algebraic surfaces. For instance, a smooth surface $X$ is rational or ruled if and only of its Kodaira dimension is $-1$ (See \cite[Theorem 6.1]{Hart}). Our goal is to compute the Kodaira dimension for smooth Nash CI surfaces. The value of the Kodaira dimension of a surface is one of the integers $-1$, $0$, $1$ or $2$.
 We say that a smooth surface is of general type if its Kodaira dimension equals $2$. We show that smooth Nash CI surface are of general type. To do so, first, we compute the canonical bundle of Nash CI varieties. From Proposition \ref{prop:eq Nash CI}, we deduce that $\Nxn$ is Gorenstein. Thus, we can use the adjunction formula (see \cite[Chapter 1.4.2]{EH}) to compute the canonical bundle of $\Nxn$.

\para 

\begin{lemma}\label{lemma:canonical nash}
For generic payoff tables, we have that 
\[
\omega_{\Nxn} = \iota^{*}\O\left(n+n_1(1-2\delta_{1,n_1})-2^{n_1},\ldots,n+n_k(1-2\delta_{1,n_k})-2^{n_k}\right),
\]
where $\iota$ is the inclusion of $\Nxn$ in $\Mn$.
\end{lemma}
\begin{proof}
Using the adjunction formula we have that 
\[
\begin{array}{c}
\omega_{\Nxn} = \iota^{*}\left(
\omega_{\Mn}\otimes \O\left(2(1-\delta_{1,n_1}),1,\ldots,1\right)
\otimes \cdots\otimes 
\O\left(1,\ldots,1,2(1-\delta_{1,n_k})\right)
\right)= \\
\iota^{*}\left(
\omega_{\Mn}
\otimes \O\left( n+n_1(1-2\delta_{1,n_1}),\ldots, n+n_k(1-2\delta_{1,n_k})\right)
\right).
\end{array}
\]
The result follows from the fact that $\omega_{\Mn}=
\O\left(-2^{n_1},\ldots,-2^{n_k}\right)$.
\end{proof}

\para 

From Lemma \ref{lemma:canonical nash}, we deduce that if $n+n_i(1-2\delta_{1,n_i})-2^{n_i}>0$ for every $i\in[k]$, then $\omega_{\Nxn} $ is ample. Hence, in this case, the Kodaira dimension will be equal to the dimension of $\Nxn$. 

\para 

\begin{corollary}\label{co: general type surface}
    Any smooth Nash CI surface has Kodaira dimension equal to $2$ and is of general type.
\end{corollary}
\begin{proof}
    In order to prove that a Nash CI surface is of general type, we need to check that the Kodaira dimension equals to $2$. Thus, it is enough to show that the canonical bundle is ample. 
 In the case of a Nash CI surface, we have that  $(n_1,\ldots,n_k) = (1,\ldots,1,2,2)$ or $(n_1, n_2) = (2,2)$ where $n \geq 4$. From Lemma \ref{lemma:canonical nash}, we deduce that 
$$\omega_{\Nxn} =\iota^{*}\O\left(
n-3,\ldots,n-3,n-2,n-2
\right) \text{ or } \omega_{\Nxn} =\iota^{*}\O\left(
1,1
\right),$$
which are ample.
\end{proof}

\para 

Note that Corollary \ref{co: general type surface} refers to smooth Nash CI surfaces.  As
exhibited in Example \ref{ex:nash variety}, there exists Nash CI surfaces that are not smooth nor irreducible.
However, we expect this behavior to be a special case and not a generic situation. 
In \cite{IJ}[Section 4] the smoothness and irreducibility of a generic Nash CI curve is derived.
In the case of surfaces, we conjecture that a generic Nash CI surface is smooth and irreducible.
This question is a more challenging problem than in the curve situation and it remains open. By Theorem \ref{theo:smooth}, we deduce that generic Nash CI surfaces are smooth away from  the hyperplanes $\{p_{j_1j_2 \cdots j_n}=0\}$ and $\{p_{++ \cdots +} = 0\}$. However, this does not give an answer to the conjecture.\\

In the next result, we analyze the connectedness of generic Nash CI varieties.

\para 

\begin{proposition}\label{prop:Nash CI connected}
Let $\mathbf{n}\neq (1,\ldots,1)$. Then, for generic payoff tables, $\Nxn$ is connected.

\end{proposition}
\begin{proof}
By Proposition \ref{prop:eq Nash CI}, we get the following exact sequence (Koszul complex):

\begin{equation}
\displaystyle
0\longrightarrow \F_n\overset{\phi_n}{\longrightarrow} \F_{n-1}\overset{\phi_{n-1}}{\longrightarrow}\cdots\overset{\phi_3}{\longrightarrow}
\F_{2}
\overset{\phi_2}{\longrightarrow}
\F_1
\overset{\phi_1}{\longrightarrow}
\O_{\Mn}\overset{\phi_0}{\longrightarrow}
\mathcal{O}_{\Nxn}\longrightarrow 0,
\end{equation}
where 
\[
\mathcal{F}_l:= \displaystyle\bigoplus_{(i_1,j_1)<\cdots<(i_{l},j_l)}^n\mathcal{O}_{\Mn}\left( -\overset{k}{\underset{l=1}{\sum}} D_{(i_l,j_l)} \right).
\]
Let $K_i$ be the kernel of $\phi_i$. Then, the above exact sequence split into the following $n$ short exact sequences:
\[
\begin{array}{ll}
 (E_0):    &  0\longrightarrow K_0\longrightarrow \O_{\Mn}\longrightarrow\O_{\Nxn}\longrightarrow 0\\
 \noalign{\vspace*{1mm}}(E_1):    &  0\longrightarrow K_1\longrightarrow \F_{1}\longrightarrow K_0\longrightarrow 0\\
 \,\,\,\,\vdots & \hspace*{2.7cm}\vdots \\
 (E_i):    &  0\longrightarrow K_i\longrightarrow \F_{i}\longrightarrow K_{i-1}\longrightarrow 0\\
  \,\,\,\,\vdots & \hspace*{2.7cm}\vdots \\
 (E_{n-1}):    &  0\longrightarrow K_{n-1}\longrightarrow \F_{n-1}\longrightarrow K_{n-2}\longrightarrow 0.\\
\end{array}
\]
In order to check $\Nxn$ is connected, we show that $h^0(\Nxn,\O_{\Nxn})=1$. From the long exact sequence of $(E_0)$, we get that
\[
H^{0}(\Mn,\O_{\Mn}) \rightarrow H^{0}(\Nxn,\O_{\Nxn})\rightarrow H^1(\Mn,K_0).
\]
Note that if $h^1(\Mn,K_0)=0$, we get a surjection 
$H^{0}(\Mn,\O_{\Mn}) \twoheadrightarrow H^{0}(\Nxn,\O_{\Nxn})$. Since $\Mn$ is connected, this would imply that $h^0(\Nxn,\O_{\Nxn})=1$. Hence, it is enough to check that $h^1(\Mn,K_0)=0$. To do so, for every $i\in[n-1]$, we consider the following exact sequence arising from $(E_i)$:
\begin{equation}\label{eq:seq connec}
H^{i}(\Mn,\mathcal{F}_{i}) \rightarrow H^{i}(\Mn,K_{i-1})\rightarrow H^{i+1}(\Mn,K_{i})\rightarrow H^{i+1}(\Mn,\mathcal{F}_{i}).
\end{equation}
We claim that $H^{l}(\Mn,\mathcal{F}_{i})= 0$ for $l\leq n$.  Indeed, each $\mathcal{F}_i$ is a direct sum of sheaves of $\O(-d_1,\ldots,-d_k)$ for some $d_i\geq 0$. Hence, it is enough to check that the corresponding cohomology groups of these sheaves vanish. Using K\"unneth formula, we obtain that 
\[
H^\alpha(\Mn,\O(-d_1,\ldots,-d_k)) = 0 
\]
for $\alpha\neq \sum_i 2^{n_i} - k$ or if some $d_i<2^{n_i}$.
Since $\mathbf{n}\neq (1,\ldots,1)$, $\sum_i 2^{n_i} - k\geq n+1$ and the equality only holds for $\mathbf{n}=(1,\ldots,1,2)$. Thus, we conclude that $H^{l}(\Mn,\mathcal{F}_{i})= 0$ for $l\leq n$. As a result, from equation \eqref{eq:seq connec}, we deduce that $H^{i}(\Mn,K_{i-1})\simeq H^{i+1}(\Mn,K_{i})$ for every $i\in[n-1]$. In particular, we get that 
\[
H^1(\Mn,K_0)\simeq H^{n}(\Mn,K_{n-1})= H^n(\Mn,\mathcal{F}_n) = 0.\qedhere
\] 
\end{proof}

\section{Universality of Nash CI varieties}\label{sec: affine universality}

In this section we study the affine universality of $\mathbf{n}$--Nash CI varieties whose partition $\mathbf{n}$ is of the form $(n_1,\ldots,n_k)=(1,\ldots,1,2,\ldots,2)$. In other words, the independence model is
\[
\Mn = \left(\P^1 \right)^{n-2l}\times \left(\P^3\right)^l.
\]
for $l\in[n]$.
In this situation, we denote the corresponding $\mathbf{n}$--Nash CI variety by $Y_{X,l}$. For instance, for generic payoff tables,  $Y_{X,0}$ corresponds to the system of multilinear equations determining the set of totally mixed Nash equilibria, $Y_{X,1}$ is the Nash CI curve, and $Y_{X,2}$ is the Nash CI surface. 
Let $U_{X,l}\subset Y_{X,l}$ be the affine open subset defined by  $$\sigma^{(1)}_2,\ldots,\sigma_2^{(n-2l)},\sigma_{22}^{(n-2l+1)},\ldots,\sigma_{22}^{(n-l)}\neq 0.$$ In this setting, the affine universality ask whether any real affine algebraic variety is isomorphic to an affine open subset of $\Nxn$ for some game. In \cite{datta}, Datta proved the affine universality for the set of totally mixed Nash equilibria, i.e. for $l=0$. In \cite{IJ}[Section 4.2], the affine universality was studied for the case of Nash CI curves.\\

In \cite{datta}, the concept of isomorphism employed for investigating the universality of Nash equilibria is specifically the notion of stable isomorphism within the category of semialgebraic sets. On the other hand, in \cite{IJ}, the notion of isomorphism used is the notion of isomorphism of algebraic varieties. Following the second approach, in this paper, we use the second notion of isomorphism. 
 As in \cite{IJ}, we rephrase \cite[Theorem 1,Theorem 6]{datta} as follows using the notion of isomorphism of algebraic varieties.

\para 

\begin{theorem}\label{theo:datta}
Let $S\subset \mathbb{R}^m$ be a real affine algebraic variety. Then, there exists an $n$-player game $X$ with binary choices such that $U_{X,0}\simeq S$.
\end{theorem}

\para 

In \cite{IJ}, the authors remarked that the affine universality does not hold for $l=1$ since the dimension of $Y_{X,1}$ is at least $1$. Similarly, 
since $Y_{X,l}$ is the intersection of $n$ divisors in $\Mn$, we have that 
\[
\dim Y_{X,l} \geq 2(n-2l)+4l-(n-l)-n = l.
\]
 Hence, $(l-1)$--dimensional real affine algebraic varieties can not be obtained from this construction, and we deduce that the affine universality does not hold for $l\geq 1$. 
 In \cite{IJ}, this dimension problem is overcome for the Nash CI curve in two different ways, giving two partial answers to the affine universality for $l=1$ in \cite[Corollary 17]{IJ} and \cite[Theorem 18]{IJ}. Our goal is to generalize these results for any $l \in \mathbb{N}$. 

 \begin{lemma}\label{lemma: univ1}
 For every $n$-player game with binary choices with payoff tables $\tilde{X}^{(1)},\ldots, \tilde{X}^{(n)}$, there exists an $(n+2l)$-player game with binary choices with payoff tables $X^{(1)},\ldots,X^{(n+2l)}$ such that $$U_{X,l}\simeq U_{\tilde{X},0}\times \mathbb{R}^l.$$
 \end{lemma}
 \begin{proof}
 Let $G_1,\ldots,G_n$ be the polynomials defining $U_{\tilde{X},0}$ in $\mathbb{A}^n$. We consider an $(n+2l)$--player game with payoff tables
 $X^{(1)},\ldots,X^{(n+2l)}$. Let $\mathbf{\tilde{n}}$ be the partition of $n+2l$ where $1$ and $2$ appear $n$ and $l$ times respectively. Let $\sigma_j^{(i)}$ for $j\in [2]$ and $i\in[n]$ be the coordinates of the $n$ factors of $\mathcal{M}_\mathbf{\tilde{n}}$ corresponding to $\P^1$, and let $\sigma_{j_1j_2}^{(n+i)}$ for $j_1,j_2\in[2]$ and $i\in[l]$ be the coordinates of the $\P^2$ factors of $\mathcal{M}_\mathbf{\tilde{n}}$. Moreover, let $F_1,\ldots,F_n,F_{n+1,1},F_{n+1,2},\ldots,F_{n+l,1},F_{n+l,2}$ be the polynomials defining $U_{X,l}$. As in the proof of \cite[Proposition 20]{IJ}, we can fix the payoff tables of the players $n+1,\ldots,n+2l$ such that 
 \[
 F_{n+i,1} = \sigma_{1,1}^{(n+i)}+ \sigma_{2,1}^{(n+i)} \text{ and }  F_{n+i,2} = \sigma_{1,1}^{(n+i)}+ \sigma_{1,2}^{(n+i)} 
 \]
for every $i\in[l]$. In particular, we get that 
\[
U_{X,l} = \bbV(F_1,\ldots,F_n)\times \mathbb{R}^l.
\]
Now, we fix the payoff tables of the first $n$ players to be 
\[
X_{j_1,\ldots,j_{n+2l}}^{(i)} = \left\{ 
\begin{array}{cc}
  \tilde{X}^{(i)}_{j_1,\ldots,j_n}   &  \text{ if } j_{n+1}=\ldots =j_{n+2l}=2\\
    0 & \text{ else }
\end{array}.
\right.
\]
Then, the polynomials $F_1,\ldots, F_n$ are equal to the polynomials $G_1,\ldots,G_n$ and we conclude that
\[
U_{X,l} =  \bbV(F_1,\ldots,F_n)\times \mathbb{R}^l \simeq U_{\tilde{X},l}\times\mathbb{R}^l.\qedhere
\popQED
\]
\end{proof}

 From Theorem \ref{theo:datta} and Lemma \ref{lemma: univ1} we deduce the following first universality theorem for Nash CI varieties.

\para 

\begin{theorem}\label{theo:univ1}
Let $l \in \mathbb{N}$ and let $S\subseteq\mathbb{R}^m$ be an affine real algebraic variety. Then, there exists $n\geq l$ and an $n$--players game $X$ with binary choices such that $U_{X,l}\simeq S\times\mathbb{R}^l$.
\end{theorem}

\para 

A consequence of Theorem \ref{theo:univ1} is that, for any $l$, the space of all varieties $Y_{X,l}$ for any binary game $X$ with any number of players satisfies Murphy's law. Indeed, we say that the space of all varieties $Y_{X,l}$ satisfies Murphy's law if, for any singularity type, there exists a game $X$ and $l\in\mathbb{N}$ such that $Y_{X,l}$ has this singularity type. Then, from the fact that $S\times \mathbb{A}^k$ has the same singularity type as $S$, we deduce that for any $l \in \mathbb{N}$ the spaces of all variety $Y_{X,l}$ satisfies Murphy's law. For further reading on Murphy's law in algebraic geometry see \cite{Vakil}.

\para 

In Theorem \ref{theo:univ1} we solved the dimension problem by artificially adding extra dimensions. In our second approach, we force the dimension to be at least $l$. 

\para 

\begin{theorem}\label{theo:univ2}
Let $l\in\mathbb{N}$ and let $S\subseteq \mathbb{R}^n$ be a real affine algebraic variety defined by  $G_1,\ldots,G_m\in\mathbb{R}[x_1,\ldots,x_n]$ with $m\leq n-l$. For every $i\in\{1,\ldots,n\}$, let $\delta_{i}$
be 
the maximum of the degrees of $x_i$ in $G_1,\ldots,G_m$.
Then, there exists a  $(\delta+n+l)-$player game with binary choices such that the affine open subset $W_X$ of $C_X$ is isomorphic to $S$, where $\delta = \delta_1+\cdots +\delta_n$.
\end{theorem}
\begin{proof}
We adapt the proofs of \cite[Theorem 6]{datta} and \cite[Theorem 18]{IJ} to our setting. Consider a $(\delta+n+l)$--players game $X$. We label the last $2l$ players by $(1,1),(1,2),\ldots,(l,1),(l,2)$. The variety $Y_{X,l}$ lies in the Segre variety $\left( \P^1\right)^{\delta+n-l}\times\left(\P^3\right)^{l}$. We denote the coordinates of the $\P^3$ factors by $\sigma_{j_1,j_2}^{(\delta+n-l+i)}$ for $j_1,j_2\in[2]$ and $i\in[l]$. Moreover, we denote the polynomials defining $U_{X,l}$ by
\[
F_1,\ldots,F_{\delta+n-l},F_{1,1},F_{1,2},\ldots,F_{l,1},F_{l,2}.
\]

As in the proof of \cite[Proposition 20]{IJ} we can fix the payoff tables of the players $(1,1),\ldots,(l,2)$ such that 
 \[
 F_{i,1} = \sigma_{1,1}^{(\delta+n-l+i)}+ \sigma_{2,1}^{(\delta+n-l+i)} \text{ and }  F_{\delta+n-l+i,2} = \sigma_{1,1}^{(\delta+n-l+i)}+ \sigma_{1,2}^{(\delta+n-l+i)} 
 \]
for every $i\in[l]$. In particular, we deduce that 
\[
\mathbb{V}(F_{1,1},\ldots,F_{l,2}) = \left(\P^1\right)^{\delta+n}.
\]
Following the proof of \cite[Theorem 6]{datta}, there exists a $(\delta+n)$--players game $\tilde{X}$ such that $U_{\tilde{X},0}= S$. Moreover, we can assume that the last $n-m$ payoff tables of the game vanish. Now, we fix the payoff tables of the first $\delta+n-l$  of the game $X$ as follows:
\[
X_{j_1,\ldots,j_{\delta+n-l},j_{1,1},\ldots,j_{l,2}}^{(i)} = \left\{ 
\begin{array}{cc}
  \tilde{X}^{(i)}_{j_1,\ldots,j_{\delta+n-l} }  &  \text{ if } j_{1,1}=\cdots =j_{l,2}=2\\
    0 & \text{ else }
\end{array}
\right.
\]
for $i\in[\delta+n-l]$. One can check that the polynomials $F_{1},\ldots, F_{\delta+n-l}$ are equal to (but with different variables) the $\delta+n-l$ polynomials defining $U_{\tilde{X},0}$. Using that $n-m\geq l$, we deduce that 
\[
U_{X,l}= \mathbb{V}(F_1,\ldots,F_{\delta+n-l})\cap \left(\P^1\right)^{\delta+n} \simeq U_{\tilde{X},0}\simeq S.\qedhere
\popQED
\]

\end{proof}

\para

\begin{remark}
    In \cite{datta}, Datta's universality theorem refers to  the set of totally mixed Nash equilibria. An analogous statement for the set of totally mixed CI equilibria can be obtained in our setting. Namely, given $l$ and a real affine algebraic variety $S$, there exists a game with binary choices such that $U_{X,l}\cap \Delta $ is isomorphic to $S\times \mathbb{R}^l$ (Corollary \ref{theo:univ1}). As in \cite{datta}, here we use the notion of stable isomorphism in the category of semialgebraic sets. To derive these results one should argue as in \cite{datta}: the set of real points of a real affine algebraic variety is isomorphic to the set of real points of a real affine algebraic variety whose real points are contained in the probability simplex. Now, assuming the latter, the statement follows from  Proposition \ref{theo:univ1}.
     An analogous statement also holds for Theorem \ref{theo:univ2}.
\end{remark}

\para 

Note that the proofs of both theorems provide a method for, given the real affine algebraic variety, finding a game satisfying the statements of the theorems.

\para

\section{Acknowledgements}
We thank Daniele Agostini, Matthieu Bouyer, Ben Hollering, Serkan Ho{\c{s}}ten, and Bernd Sturmfels for helpful discussions that greatly benefited this project. Javier Sendra–Arranz received the support of a fellowship from the “la Caixa” Foundation
(ID 100010434). The fellowship code is LCF/BQ/EU21/11890110.

\addcontentsline{toc}{section}{References}\label{sec:references}
\bibliographystyle{plain}

\end{document}